\documentclass[11pt]{amsart}
\usepackage{amsmath, amsthm, amssymb, amsfonts}
\usepackage[normalem]{ulem}
\usepackage[backref]{hyperref}

\usepackage{xypic, mathtools}
\usepackage{verbatim} 
\usepackage{longtable} 

\usepackage{mathtools}

\usepackage{caption}

\usepackage{tikz-cd}
\usetikzlibrary{arrows}

\theoremstyle{plain}
\newtheorem{thm}{Theorem}
\newtheorem{cor}{Corollary}
\newtheorem{lemma}{Lemma}
\newtheorem{lem}[lemma]{Lemma}

\newtheorem{prop}{Proposition}
\newtheorem{conj}{Conjecture}

\theoremstyle{definition}
\newtheorem{defn}{Definition}

\theoremstyle{remark}

\usepackage{tikz-cd}
\usetikzlibrary{arrows}

\newcommand{\BC}{{\mathbb{C}}}

\newcommand{\BL}{{\mathbb{L}}}

\newcommand{\BP}{{\mathbb{P}}}
\newcommand{\BQ}{{\mathbb{Q}}}

\newcommand{\BZ}{{\mathbb{Z}}}

\newcommand{\CA}{{\mathcal A}}

\newcommand{\CE}{{\mathcal E}}
\newcommand{\CF}{{\mathcal F}}

\newcommand{\CI}{{\mathcal I}}

\newcommand{\CL}{{\mathcal L}}
\newcommand{\CM}{{\mathcal M}}
\newcommand{\CN}{{\mathcal N}}
\newcommand{\CO}{{\mathcal O}}

\newcommand{\CS}{{\mathcal S}}

\newcommand{\CU}{{\mathcal U}}
\newcommand{\CV}{{\mathcal V}}

\newcommand{\CX}{{\mathcal X}}
\newcommand{\CY}{{\mathcal Y}}
\newcommand{\CZ}{{\mathcal Z}}

\newcommand{\blangle}{\big\langle}

\DeclareMathOperator{\Hilb}{Hilb}

\DeclareFontFamily{OT1}{rsfs}{}
\DeclareFontShape{OT1}{rsfs}{n}{it}{<-> rsfs10}{}
\DeclareMathAlphabet{\curly}{OT1}{rsfs}{n}{it}

\newcommand\Spec{\operatorname{Spec}}

\newcommand*\dd{\mathop{}\!\mathrm{d}}

\newcommand{\Coh}{\mathrm{Coh}}
\newcommand{\Pic}{\mathop{\rm Pic}\nolimits}
\newcommand{\PT}{\mathsf{PT}}
\newcommand{\DT}{\mathsf{DT}}

\usepackage{tikz}
\usepackage{lmodern}
\usetikzlibrary{decorations.pathmorphing}
\tikzset{snake it/.style={decorate, decoration=snake}}

\begin{document}
\title[Reduced Dondaldson-Thomas invariants]{
Reduced Donaldson--Thomas invariants
and the ring of dual numbers
}
\date{\today}

\author{Georg Oberdieck}
\address{MIT, Department of Mathematics}
\email{georgo@mit.edu}

\author{Junliang Shen}
\address{ETH Z\"urich, Department of Mathematics}
\email{junliang.shen@math.ethz.ch}

\begin{abstract}
Let $A$ be an abelian variety. We introduce
$A$-equivariant Grothendieck rings
and $A$-equivariant motivic Hall algebras,
and endow them with natural
integration maps to the ring of dual numbers.
The construction allows a systematic treatment
of reduced Donaldson--Thomas invariants
by Hall algebra techniques.
We calculate reduced Donaldson--Thomas invariants
for $\mathrm{K3} \times E$ and abelian threefolds
for several imprimitive curve classes.
This verifies (in special cases) multiple cover
formulas conjectured by
Oberdieck--Pandharipande and
Bryan--Oberdieck--Pandharipande--Yin.
\end{abstract}
\baselineskip=14.5pt
\maketitle

\setcounter{tocdepth}{1} 

\tableofcontents

\section{Introduction}
\subsection{Equivariant Hall algebras}
We present a framework
to apply techniques from motivic
Hall algebras and Grothendieck rings of varieties
in the presence of an action by an abelian variety.
The idea is to incorporate the action as \emph{additional data}
into the definition, making the Hall algebra
and the underlying Grothendieck groups equivariant.
The natural integration map by Euler characteristic is replaced
by an integration map to the ring of dual numbers:
\[ \BQ[\epsilon] / ( \epsilon^2 = 0 )\,. \]
Precisely, given a scheme $Z$ with an action by
a simple\footnote{An abelian variety is simple if all its proper subgroups are $0$-dimensional.} abelian variety $A$
we define the integration map by
\[ \CI(Z) = e(Z^A) + e( (Z - Z^A)/A ) \epsilon \]
where $Z^A$ is the fix locus of the action,
and $e( \cdot )$ is the topological Euler characteristic
taken here always in the orbifold sense.
This construction arises natural in applications.
For example, for a smooth
projective variety $X$ of dimension $d$
we have the identity in the Grothendieck ring of varieties
\begin{equation} \sum_{n = 0}^{\infty} \, [ \Hilb^n(X) ]\, q^n
\, = \,
\left( \sum_{n = 0}^{\infty} \, [ \Hilb^n( \BC^d )_0 ] \, q^n \right)^{[X]} \label{3513} \end{equation}
where $\Hilb^n(X)$ is the Hilbert scheme of points on $X$,
and $\Hilb^n( \BC^d )_0$ is the punctual Hilbert scheme
in the affine space $\BC^d$ \cite{GLM2}.
In case $X = Y \times A$ where $A$ is a simple abelian variety
acting on $X$ by translation in the second factor,
a straight-forward argument shows that \eqref{3513}
lifts to the $A$-equivariant Grothendieck ring.
Applying our integration map
we naturally obtain\footnote{See Section 3 for details of the proof.}
\begin{equation}
\begin{aligned} \label{eqn2}
1 + \epsilon \sum_{n=1}^{\infty}
e\left( \Hilb^n(X) / A \right) q^n
& = 
\left( \sum_{n = 0}^{\infty} e\big(\Hilb^n( \BC^d )_0\big) q^n \right)^{\CI(X)} \\
& = 
\left( \sum_{n = 0}^{\infty} P_d(n) q^n \right)^{\epsilon \cdot e(Y)} \\
& = 1 + \epsilon \cdot e(Y) \log\left( \sum_{n = 0}^{\infty} P_d(n) q^n \right)\,,
\end{aligned}
\end{equation}
where $P_d(n)$ is the number of $d$-dimensional partitions of $n$,
and we used the convention $f(q)^{\epsilon} = \exp( \log(f) \epsilon)$.
The left hand side is (up to a factor)
the generating series of Euler characteristics of
the generalized Kummer schemes of $Y \times A$,
and we recover a formula proven by
Shen \cite{S}, Morrison-Shen \cite{MS},
and Gulbrandsen-Ricolfi \cite{GR}.
In fact, the first order expansion in terms of $e(A)$
was the main motivation that led Gulbrandsen to conjecture \eqref{eqn2}
for abelian varieties in \cite{Gul}.
Our approach captures this intuition
and makes it mathematically rigorous.

\subsection{Reduced Donaldson--Thomas invariants}
Our main interest here lies in
applications to Donaldson--Thomas (DT) invariants\footnote{Donaldson--Thomas invariants are defined by integration over the Hilbert scheme
of curves in threefolds and virtually enumerate algebraic curves,
see \cite{PT3} for an introduction.}
for special Calabi--Yau threefolds $X$.
We are particularly interested in the following examples:
\begin{itemize}
\item[(1)] $X$ is an abelian threefold, or
\item[(2)] $X$ is the product of a K3 surface and an elliptic curve $E$.
\end{itemize}
In both cases an abelian variety
acts on the Hilbert schemes by translation
and
forces almost all ordinary DT invariants to vanish.
The definition of DT invariants needs to be modified
to be enumerative meaningful.

Let $A$ be an abelian variety which acts on a
Calabi--Yau threefold $X$.
Let $\Hilb^n(X, \beta)$ be the Hilbert scheme of $1$-dimensional
subschemes ${Z \subset X}$ satisfying
\[ [Z] = \beta \, \in H_2(X, \BZ), \ \quad \chi( \CO_Z ) = n \in \BZ \,. \]
If the induced $A$-action on $\Hilb^n(X, \beta)$
has finite stabilizers,
we define \emph{$A$-reduced Donaldson--Thomas invariants} of $X$ by
\begin{equation} \label{eqn3}
\DT^{X, A\text{-red}}_{n, \beta}
\, = \, \int_{\Hilb^n(X, \beta)/A} \nu \dd{e}
\, = \, \sum_{k \in \BZ} e\left( \nu^{-1}(k) \right) \,,
\end{equation}
where $\nu : \Hilb^n(X, \beta)/A \to \BZ$
is Behrend's constructible function \cite{B}.

For abelian threefolds (acting on itself by translation)
the definition was introduced by Gulbrandsen in \cite{Gul},
where he also showed deformation invariance in many cases.
For $\mathrm{K3} \times E$ 
the definition is by Bryan \cite{Bryan-K3xE} and
deformation invariance is proven in \cite{O1}.
In both cases
explicit conjectural formulas for the
reduced DT invariants
are known in all
curve classes
\cite{K3xE, BOPY}.
The formulas reveal
(at least conjecturally and as far as numbers go)
rich structures underlying
the enumerative geometry of algebraic curves.

In Section~\ref{Section_Equivariant_motivic_Hall_algebras}
we introduce $A$-equivariant
versions of Joyce's motivic Hall algebra
and equip them with integration maps defined over the ring of dual numbers.
This structure is tailored
to deal with reduced DT invariants systematically.
This leads to new calculations
in several interesting cases,
and to DT/PT correspondences in previous unknown cases.

\subsection{Reduced DT invariants for $K3\times E$}
Let $S$ be a non-singular projective K3 surface
and let $E$ be a non-singular elliptic curve. We consider the
product Calabi--Yau
\[ X = S \times E \]
on which $E$ acts by translation in the second factor.
Using the K\"unneth decomposition we identify
\[ H_2(X, \BZ) = H_2(S, \BZ) \oplus H_2(E, \BZ)
= H_2(S, \BZ) \oplus \BZ \,. \]
The conjectural form of the reduced DT invariants of
$X$ is reviewed in Section \ref{Section_Multiple_cover_formulas};
here we prove the following special case.
Define coefficients $\mathsf{m}(d,n)$ by the expansion
\[ \sum_{d = 0}^{\infty} \sum_{n \in \BZ} \mathsf{m}(d,n) p^n t^{d}
=
-\frac{24 \wp(p,t)}{\prod_{m \geq 1}(1-t^m)^{24}}
\]
where $\wp$ is the Weierstra{\ss} elliptic function,
\begin{equation} \label{Weierstrass}
\wp(p,t)
= \frac{1}{12} + \frac{p}{(1-p)^2}
+ \sum_{d = 1}^{\infty} \sum_{k|d} k (p^k - 2 + p^{-k}) t^{d} \,.
\end{equation}

\begin{thm} \label{K3xE_thm} For all $d \geq 0$ we have
\begin{equation*}
\exp\left( \sum_{n = 1}^{\infty} \DT^{X, E\textup{-red}}_{n, (0,d)} (-p)^n \right)
=
\prod_{\ell = 1}^{\infty}
\left( \frac{1}{1-p^{\ell}} \right)^{\mathsf{m}(d,\ell)} \,.
\end{equation*}
\end{thm}
\vspace{3pt}

Theorem~\ref{K3xE_thm}
determines all reduced
invariants in classes $(0,d)$.\footnote{
The Hilbert scheme $\Hilb^n(X,(0,d))$ is empty
for $n<0$ and $E$-invariant for $n=0$.}
As in the case of the Hilbert scheme of points
we need to exponentiate
the generating series of reduced DT invariants
to obtain product expansions.
The case $d = 0$ of Theorem~\ref{K3xE_thm}
recovers the calculation of reduced degree~$0$ DT invariants
of \cite{S, MS}. For $d > 0$ the results give
a new and non-trivial check in imprimitive classes
for the general multiple cover formula conjectured in \cite{K3xE}.
Explicitly, taking the logarithm in the theorem yields the closed formula
\[ \DT^{X, E\text{-red}}_{n, (0,d)}
= (-1)^n \sum_{k | (n,d)} \frac{1}{k} \mathsf{m}\left(d, \frac{n}{k} \right) \,. \]

\subsection{Reduced DT invariants for abelian threefolds}
Let $A$ be a projective abelian threefold
acting on itself by translation.
If $n \neq 0$ by deformation invariance
the $A$-reduced DT
invariants
depend only on the \emph{type}\footnote{The type
is obtained from the standard divisor theory of the dual
abelian variety \cite{BOPY}.}
\[ (d_1, d_2, d_3),\ d_i \geq 0 \]
of the curve class $\beta$. We write
\[
\DT^{\text{red}}_{n, (d_1, d_2, d_3)} = \DT^{A, A\text{-red}}_{n, \beta} \,. \]
We restrict here to the degenerate case where
$\beta$ has type $(0,0,d)$.
If $n < 0$ the Hilbert scheme is empty
and all reduced invariants vanish.
For $n = 0$ $A$-reduced invariants are not defined.
For $n > 0$ we have the following result.

\begin{thm} \label{Thm1}
For all $d\geq 0$ and $n>0$ we have
\[
\DT^{\mathrm{red}}_{n,(0,0,d)}
= \frac{(-1)^{n-1}}{n} \sum_{k|\mathrm{gcd}(n,d)} k^2.
\]
\end{thm}
In case $d=0$ the above formula
specializes to the degree 0 reduced DT invariants
which were conjectured in \cite{Gul} and proven in \cite{S, MS, GR}
using generalized Kummer schemes.
If $d > 0$ we obtain agreement
with the multiple cover formulas of \cite{BOPY},
compare Section~\ref{Section_Multiple_cover_formulas}.

\subsection{Reduced DT/PT correspondence}
A stable pair on a threefold $X$ is the datum $(\CF,s)$ of a
pure $1$-dimensional sheaf $\CF$ and a section $s \in H^0(\CF)$
with $0$-dimensional cokernel. Following \cite{PT1} we let
$P_n(X ,\beta)$ be the moduli space of stable pairs
with numerical invariants
\[
[\mathrm{Supp}(\CF)]=\beta \in H_2(X,\BZ),
\ \quad \chi(\CF)=n \in \BZ \,. \]
Let $A$ be an abelian variety which acts
on a Calabi--Yau threefold $X$.
If the induced action on $P_n(X,\beta)$
has finite stabilizers,
we define \emph{$A$-reduced Pandharipande-Thomas (PT) invariants} by
\[
\PT^{X, A\text{-red}}_{n,\beta}
= \int_{P_n(X,\beta)/A} \nu \dd{e},
\]
where $\nu : P_n(X, \beta)/A \to \BZ$ is the Behrend function.

The relationship between usual 
DT and PT invariants of Calabi--Yau 3-folds
has been well understood
via wall-crossing \cite{T10,Br1,T16}.
For abelian threefolds $A$
we prove in Section~\ref{Subsection_Abelianthreefold_DTPT}
the following simple correspondence:
\[
\DT^{A, A\text{-red}}_{n,\beta}
= \PT^{A, A\text{-red}}_{n,\beta}
\]
for all $n, \beta$ where $A$-reduced invariants are defined.

For $E$-reduced invariants of $\mathrm{K3} \times E$ the DT/PT correspondence
takes a more interesting form.
Define generating series of reduced invariants:
\[
\DT^{\text{red}}_d(q,t) =
\sum_{n, \gamma} \DT^{X,E\text{-red}}_{n,(\gamma,d)} q^n t^{\gamma},
\quad
\PT^{\text{red}}_d(q,t)
= \sum_{n, \gamma} \PT^{X,E\text{-red}}_{n, (\gamma,d)} q^n t^{\gamma}
\]
where the sums run over all $n \in \BZ$
and all curve classes $\gamma \in H_2(S,\BZ)$
with $(n,\gamma) \neq 0$.
Let also
\[ M(q) = \prod_{m = 1}^{\infty} (1-q^n)^{-n} \]
be the MacMahon function, and
define coefficients $a_d$ by
\[ \sum_{d = 0}^{\infty} a_d t^d =
\prod_{m = 1}^{\infty} (1-t^m)^{-24} \,. \]
By a result of G\"ottsche \cite{Goe} we have $a_d = e(\Hilb^d(S))$.

\begin{thm} \label{Thm_DTPT} For all $d \geq 0$,
\[ \exp \left( \DT^{\mathrm{red}}_d(q,t) \right)
\, = \, 
M(-q)^{-24 a_d} \cdot \exp\left( \PT^{\mathrm{red}}_d(q,t) \right) \,.
\]
\end{thm}

If $\gamma \neq 0$ then
we recover the result of \cite{O1},
\[
\DT^{X, E\text{-red}}_{n,(\gamma,d)}
= \PT_{n,(\gamma,d)}^{X, E\text{-red}} \,,
\]
while for $\gamma = 0$ the correspondence
(Theorem~\ref{Thm_DTPT}) is new and non-trivial.

\subsection{Relation to other work}
(1) The motive of the generalized Kummer schemes
were computed in \cite{MS}
using Grothendieck rings relative to an abelian monoid.
It would be interesting to
compare this to the motivic class
(in the Grothendieck ring of stacks)
of the stack quotient
$\Hilb^n(X) / A$.

\noindent (2) The topological vertex method
of \cite{BrKo,Bryan-K3xE} may yield another
approach to Theorems \ref{K3xE_thm} and \ref{Thm1}.
The method proceeds by stratification and
computation of the local invariants.
While in principle this method
is able to compute the Euler characteristic
of the corresponding Hilbert scheme,
the difficulty here is
to incorporate also the correct Behrend function
weights into the computation for DT invariants.

\subsection{Plan of the paper}
In Section~\ref{Section_Multiple_cover_formulas}
we recall the general multiple cover
formulas for abelian threefolds and $\mathrm{K3} \times E$
as conjectured in \cite{K3xE, BOPY}.
We also comment on the relationship of Theorem~\ref{K3xE_thm}
to Gromov-Witten theory.
In Section~\ref{Section_Equivariant_Grothendieck_rings}
as warmup for the general case
we introduce an equivariant Grothendieck ring of varieties and
prove the degree $0$ cases of Theorems \ref{K3xE_thm} and \ref{Thm1}.
In Section~4 we introduce the equivariant motivic Hall algebra,
which we apply in Section~\ref{Section_Reduced_DT_for_K3xE}
to prove the main theorems following a strategy of Y.~Toda \cite{T12}.
In Section~6 we treat the parallel case of abelian threefolds.

\subsection{Conventions}
We always work over the complex numbers $\BC$.
All schemes are of finite type, and by definition
a variety is a reduced, separated scheme of finite type.
A Calabi--Yau threefold is a nonsingular
projective threefold $X$ with trivial canonical class
$K_X \simeq \CO_X$.
In particular the vanishing of $H^1(X, \CO_X)$
is not required.
By the recent work \cite{PTVV, T16}
the results of \cite[Sec.5]{Br2} also hold in this
more general setting, compare \cite[4.6]{O1}.


\subsection{Acknowledgements}
The paper was started
when J.~ S. was visiting
MIT in September 2016.
We would like to thank Jim Bryan,
Andrew Kresch, Davesh Maulik,
Rahul~Pandharipande, Johannes Schmitt, and Qizheng Yin
for their interest and useful discussions.

J.~ S. was supported by grant ERC-2012-AdG-320368-MCSK in the group of Rahul Pandharipande at ETH Z\"{u}rich.

\section{Multiple cover formulas}
\label{Section_Multiple_cover_formulas}
\subsection{Overview}
We review here the conjectural formulas
for reduced DT invariants of
$\mathrm{K3} \times E$ by \cite{K3xE}
and abelian threefolds by \cite{BOPY}.

\subsection{$K3 \times E$}
Let $X = S \times E$ be the product
of a non-singular projective
K3 surface $S$ and an elliptic curve $E$,
on which $E$ acts by translation in the second factor.
The $E$-reduced DT invariants of $X$ are denoted by
\[ \DT^{\mathrm{red}}_{n, (\beta,d)} = \DT^{X, E\text{-red}}_{n,(\beta,d)} \]
where $\beta \in H_2(S, \BZ)$ is a (possibly zero) curve class,
$d \geq 0$ and $n \in \BZ$.
Since we require the translation action
on the Hilbert scheme
to have finite stabilizers
we will always require 
\[ \beta \neq 0 \ \ \text{or} \ \ n \neq 0 \,. \]

Define coefficients $c(m)$ by the expansion
\[
\sum_{d \geq 0} \sum_{k \in \BZ} c(4d-k^2) p^k t^d
=
24 \phi_{-2,1}(p,t) \wp(p,t)
\]
where $\phi_{-2,1}$ is
the unique weak Jacobi form of index $1$ and weight $-2$,
\begin{equation} \label{phim21}
\phi_{-2,1}(p,t)
=
(p - 2 + p^{-1})
\prod_{m \geq 1} \frac{ (1-pt^m)^2 (1-p^{-1}t^m)^2}{ (1-t^m)^4 }
\end{equation}
and $\wp$ is the Weierstra{\ss} elliptic function \eqref{Weierstrass}.
The weight $10$ Igusa cusp form is defined by the product expansion
\[ \chi_{10}(p,t, \tilde{t}) \ = \ p \, t \, \tilde{t} \prod_{k,h,d} (1-p^k t^h \tilde{t}^d)^{c(4hd-k^2)}
\]
where the product is over all
$k \in \BZ$ and $h,d \geq 0$ such that
\begin{enumerate}
\item[$\bullet$]
 $h>0$ or $d>0$,
\item[$\bullet$]
  $h = d = 0$ and $k < 0$ \,.
\end{enumerate}
We define coefficients $\mathsf{m}(h,d,n)$
by the expansion of the reciprocal
of the Igusa cusp form
in the region $0 < |t| < |p| < 1$,
\[
\sum_{h = 0}^{\infty} \sum_{d = 0}^{\infty} \sum_{n \in \BZ}
\mathsf{m}(h,d,n) p^n t^{h-1} \tilde{t}^{d-1}
\ = \ 
\frac{-1}{\chi_{10}(p, t, \tilde{t})} \,.
\]
The coefficients $\mathsf{m}(h,d,n)$ are related
to $\mathsf{m}(d,n)$ introduced before by
\[
\mathsf{m}(d,n) = \mathsf{m}(1,d,n) \,.
\]

The following conjecture was proposed in \cite{K3xE}.
\begin{conj}[\cite{K3xE}] \label{K3xE_conj}
For all $n, \beta,d$ satisfying $\beta \neq 0$ or $n \neq 0$,
we have
\begin{equation} \label{K3xE_MC}
(-1)^n \DT^{\mathrm{red}}_{n, (\beta,d)}
=
\sum_{\substack{k \geq 1 \\ k | (n, \beta)}}
\frac{1}{k} \mathsf{m}\left( \frac{(\beta/k)^2}{2} + 1,\, d,\,  \frac{n}{k} \right)
\end{equation}
where $\gamma^2 = \gamma \cdot \gamma$
is the self-intersection of a class $\gamma \in H_2(S,\BZ)$.
\end{conj}

The equality of Conjecture~\ref{K3xE_conj} is conjectured
to hold for all cases where it is defined.
Indeed, the reduced DT invariants on the left hand side
are defined if and only if $(\beta,n) \neq (0,0)$
which precisely coincides with the case
where the sum on the right hand side makes sense.

If $\beta$ is primitive of square $\beta^2 = 2h-2$
then \eqref{K3xE_MC} says
the reduced DT invariant
is up to a sign equal to the coefficient $\mathsf{m}(h,d,n)$.
If $\beta$ is imprimitive, then \eqref{K3xE_MC}
expresses the reduced DT invariant in terms of primitive invariants.
Hence we sometimes refer to \eqref{K3xE_MC}
as a multiple cover formula.
In the most degenerate case $\beta = 0$ we recover Theorem~\ref{K3xE_thm}.

Finally, for every $d \geq 0$ the rule \eqref{K3xE_MC}
may be reformulated in the following product expansion:
\[
\exp\bigg(
\sum_{(n, \beta) \neq 0}
\DT^{\mathrm{red}}_{n, (\beta,d)} (-p)^n t^{\beta} 
\bigg)
=
\prod_{(\ell, \gamma) \neq 0}
\left(
\frac{1}{1 - p^{\ell} t^{\gamma}}
\right)^{\mathsf{m}(\gamma^2/2+1,d,\ell)}
\]
where $(n,\beta)$ and $(\ell, \gamma)$
run over all non-zero pairs of an integer and a
(possibly zero) curve class in $H_2(S,\BZ)$.

\subsection{Comparision with Gromov--Witten theory}
The formula \eqref{K3xE_MC}
was conjectured in \cite{K3xE}
for reduced Gromov--Witten (GW) invariants
in curve classes $(\beta,d)$ where $\beta \neq 0$.
Translating the statement to DT theory
via the conjectural reduced GW/DT correspondence\footnote{
The reduced GW/PT correspondence is conjectured in \cite[Conj.D]{K3xE},
to which we apply the $\DT/\PT$-correspondence of \cite{O1}.}
yields Conjecture~\ref{K3xE_conj}.
While reduced GW invariants are not defined for
$\beta = 0$,
the formula makes sense on the DT side
and surprisingly gives the correct result. 

If $\beta$ vanishes the Donaldson--Thomas generating
series is not a rational function
and the variable change $p=e^{iu}$ is not well-defined.
However, parallel to the case of degree zero
DT invariants discussed in \cite[2.1]{MNOP}
an \emph{asymptotic} correspondence
may be established as follows.

The analog of the reduced (disconnected) Gromov--Witten potential
in case $\beta = 0$ and genus $g \geq 2$ is the series
\[
\CF^g(t) = \int_S c_2(S) \cdot \prod_{m \geq 1} \frac{1}{(1-t^m)^{\int_S c_2(S)}}
\sum_{d = 0}^{\infty}
\frac{1}{2g-2} \blangle \tau_1(\omega) \lambda_{g-1} \lambda_{g-2} \rangle^{E}_{g, d}t^d
\]
where $\langle \cdot \rangle^E_{g, d}$
are the connected Gromov--Witten
invariants of the elliptic curve~$E$
in genus $g$ and degree $d$,
and $\omega \in H^2(E, \BZ)$ is the class of a point,
$\tau_1$ is the first descendent insertion,
and $\lambda_{k}$ is the $k$-th Chern class of the Hodge bundle.
The Euler factor
\[ \prod_{m \geq 1} \frac{1}{(1-t^m)^{\int_S c_2(S)}} \]
is the contribution of the
non-reduced Gromov--Witten theory of $X$. 
The factor $2g-2$ corrects for the
integration of the cotangent line bundle over each curve,
compare \cite[Sec.7]{BOPY}.
A calculation by Pixton \cite[Prop.4.4.6]{Pix}
based on the results \cite{OP1, OP3}
shows
\[ \sum_{d = 0}^{\infty}
\blangle \tau_1(\omega) \lambda_{g-1} \lambda_{g-2} \rangle^{E}_{g, d}t^d
=
(-1)^{g} B_{2g-2} \binom{2g}{2} C_{2g}(t) \]
where $B_{k}$ are the Bernoulli numbers and
\[
C_{k}(t) = -\frac{B_k}{k \cdot k!}
+ \frac{2}{k!}
\sum_{n \geq 1} \sum_{\ell | n} \ell^{k-1} t^n \]
are renormalized classical Eisenstein series.
Let also
\[ \CF^g(t) = \sum_{d = 0}^{\infty} \CF^g_d t^d \,. \]

Then by Theorem~\ref{K3xE_thm}
the asmptotic Gromov--Witten/Donaldson--Thomas correspondence
holds for all $d \geq 0$:
\begin{equation} \label{corrrr}
\sum_{g = 2}^{\infty} \CF_d^g u^{2g-2}
\, \sim \,
c_d \cdot
\sum_{n=1}^{\infty} \DT^{\mathrm{red}}_{n,(0,d)} (-p)^n
\end{equation}
under the variable change $p=e^{iu}$,
where we have $c_0 = -1/2$, and $c_d = -1$ for all $d \geq 1$,
and $\sim$ stands for taking the formal expansion on the right hand side,
interchanging sums and renormalizing the genus $g \geq 2$ terms
via negative zeta values.
The overall minus sign in the correspondence~\eqref{corrrr}
corresponds to the difference of the Behrend function
of the Hilbert scheme and its quotient by translation.
The factor $1/2$ in case $d=0$ is
parallel (via taking the logarithm)
to the square root
in the degree 0 asymptotic GW/DT correspondence \cite[Eqn.2]{MNOP}.

\subsection{Abelian 3-folds}
Let $A$ be an abelian threefold, and let
$\beta \in H_2(A, \mathbb{Z})$
be a curve class of type $(d_1, d_2, d_3)$.
Assuming deformation invariance also in the case
$n=0$ we will simply write
\[ \DT^{\mathrm{red}}_{n, (d_1, d_2, d_3)}
= \DT^{A, A\text{-red}}_{n, (d_1, d_2, d_3)} \]
The translation action on the Hilbert scheme
has finite stabilizers
(and hence reduced DT invariants are defined)
if and only if $n \neq 0$ or at least
two of the integers $d_1, d_2, d_3$ are positive.

Define coefficients $\mathsf{a}(k)$ by the expansion
\[
\sum_{d = 0}^{\infty} \sum_{r \in \BZ} \mathsf{a}(4d-r^2) p^r t^d
\, = \, 
- \phi_{-2,1}(p,t)
\]
where the Jacobi form $\phi_{-2,1}$
was defined in \eqref{phim21}. Let also
\[ \mathsf{n}(d_1,d_2,d_3,k) = \sum_{\delta} \delta^2 \]
where $\delta$ runs over all divisors of 
\[ \gcd\left( k, d_1, d_2, d_3, \frac{d_1 d_2}{k}, \frac{d_1 d_3}{k}, \frac{d_2 d_3}{k}, \frac{d_1 d_2 d_3}{k^2} \right) \, \]
when all numbers in the bracket are integers.

\begin{conj}[\cite{BOPY}] \label{Ab_conj}
If $n > 0$ or
at least two of the $d_i$ are positive, then
\[ (-1)^n \DT^{\mathrm{red}}_{n,(d_1,d_2,d_3)}
= \sum_{k} \frac{1}{k}\,
\mathsf{n}(d_1,d_2,d_3,k) \cdot
\mathsf{a}\left( \frac{4 d_1 d_2 d_3 - n^2}{k^2} \right)
\]
where $k$ runs over all divisors of
$\gcd( n, d_1 d_2, d_1 d_3, d_2 d_3 )$ such that $k^2 \big| d_1d_2d_3$.
\end{conj}

For abelian threefolds
we obtain product formulas only if $d_1 = 1$ (up to permutation).
Assuming Conjecture~\ref{Ab_conj}
we have in analogy with the Igusa cusp form
\[
\exp\left( \sum_{d, \tilde{d} = 0}^{\infty} \sum_{n \in \BZ}
\DT^{\mathrm{red}}_{n, (1, d, \tilde{d})} (-p)^n t^{d} \tilde{t}^{\tilde{d}} \right)
=
\prod_{h, d, k}
\left( \frac{1}{1 - p^k t^{h} \tilde{t}^{d}} \right)^{\mathsf{a}(4 h d - k^2)}
\]
where the product is over all
$k \in \BZ$ and $m_1, m_2 \geq 0$ such that
$m_1 > 0$, or $m_2 > 0$, or $m_1 = m_2 = 0$ and $k > 0$.

\section{Equivariant Grothendieck rings}
\label{Section_Equivariant_Grothendieck_rings}
\subsection{Overview}
As a toy example for the equivariant Hall algebra
we introduce the equivariant Grothendieck ring
and its integration map to the dual numbers.
As application we reprove the following result of \cite{S}
and \cite{MS}.

Let $A$ be an abelian variety and
let $Y$ be a non-singular quasi-projective variety.
The action of $A$ act on $Y \times A$ by translation in the second factor induces an action
on the  Hilbert scheme of points
$\mathrm{Hilb}^n(Y \times A)$ by translation.
The quotient
\[
\mathrm{Hilb}^n(Y \times A)/A
\]
is a Deligne--Mumford stack for every $n>0$. We also let $d = \dim( Y \times A )$.

\begin{thm} \label{Thm_degree0_Euler_Characteristic} We have
\[ \mathrm{exp}\left(\sum_{n = 1}^{\infty}
e\big( \mathrm{Hilb}^n(Y \times A)/A \big)
q^n \right)
= \left(\sum_{n = 0}^{\infty} P_{d}(n)q^n\right)^{e(Y)}.
\]
where $P_d(n)$ is the number of $d$-dimensional partitions of $n$.
\end{thm}

\subsection{Equivariant Grothendieck rings}
Let $A$ be a simple
abelian variety of dimension $g > 0$.
The $A$-equivariant Grothendieck group of varieties
is the free abelian group $K^A_0(\mathrm{Var})$
generated by the classes
\[
[X , a_X]
\]
of a variety $X$ together with an $A$-action $a_X : A \times X \rightarrow X$,
modulo the equivariant scissor relations:
For every $A$-invariant closed sub-variety
$Z \subset X$ with complement $U$,
\[
[X , a_X] = [Z, a_X|_Z] + [U, a_X|_U] \,.
\]

For varieties $X$ and $Y$
with $A$-actions $a_X$ and $a_Y$ respectively,
let $a_{X \times Y}$ be the $A$-action on the product $X \times Y$
obtained from the diagonal $A \to A \times A$
and the product action $a_X \times a_Y$.
We define a multiplication on $K^A_0(\mathrm{Var})$ by
\[
[X, a_X] \times [Y,a_Y] = [X \times Y , a_{X \times Y}] \,.
\]
The product is commutative and associative with unit 
\[ [ \mathrm{Spec(\mathbb{C})}, a_{\mathrm{triv}} ] \]
where $a_{\mathrm{triv}}$ is the trivial $A$-action.
We call the pair $\big( K^A_0(\mathrm{Var}), \times \big)$
the $A$-equivariant Grothendieck ring. 

\subsection{Schemes}
The $A$-equivariant Grothendieck group of schemes
is the free abelian group $K^A_0(\mathrm{Sch})$
generated by the classes
$[X , a_X]$
of a scheme $X$ together
with an $A$-action $a_X : A \times X \rightarrow X$,
modulo the following relations:
\begin{enumerate}
\item[(a)] $[X \sqcup Y, a_X \sqcup a_Y] = [X,a_X]
+ [Y, a_Y]$ for every pair of schemes $X$ and $Y$
with $A$-actions $a_X$ and $a_Y$ respectively,
\item[(b)] $[X, a_X] = [Y, a_Y]$ for every
$A$-equivariant geometric bijection\footnote{The map
$f$ is a geometric bijection if
the induced map
$f(\BC) : X(\BC) \to Y(\BC)$
on $\BC$-valued points is a bijection,
see \cite[Defn.2.7]{Br2}.}
${X \xrightarrow{f} Y}$.
\end{enumerate}
The product on $K^A_0(\mathrm{Sch})$
is defined identical to the case of varieties.
Since the equivariant scissor relation
is implied by relations (a) and (b)
the natural embedding of the category
of varieties into the category of schemes
determines a ring homomorphism
\begin{equation} K^A_0(\mathrm{Var}) \to K^A_0(\mathrm{Sch}) \,.
\label{23313}
\end{equation}

\begin{lemma}\label{2333} The morphism \eqref{23313} is an isomorphism.
\end{lemma}
\begin{proof}
This is parallel to \cite[Sec.2.3, 2.4]{Br2}.
Let $X$ be a scheme with $A$-action $a_X$.
We first show the class $[X,a_X]$ is in the image of \eqref{23313}.

By relation (b) we may assume $X$ is reduced.
Then there is an affine open $U \subset X$
such that every point of $u$ is seperated in $X$.\footnote{
Let $\Delta \subset X \times X$ be the diagonal.
The non-separated points of $X$ are the closure
the image of $\overline{\Delta} \setminus \Delta$
under the projection to the second factor.
Hence we may assume $X$ is irreducible.
Since $\overline{\Delta} \setminus \Delta$
has dimension strictly less then $X$,
the scheme $\overline{\Delta} \setminus \Delta$ does not dominate $X$.}
By the valuative criterion, being seperated
is invariant under translation by $A$.
Hence every point of
the translate $U+A$, i.e.
the image of $A \times U \xrightarrow{a_X} X$,
is seperated, and $U+A$ is a variety.
Repeating the argument with the complement of $U+A$,
by induction there exists an $A$-equivariant stratification of $X$ by
varieties. Thus $X$ admits an $A$-equivariant geometric bijection
from a variety $Y$, which by (b) implies the claim.

It remains to check the relations imply each other.
The key step is to prove relation (b)
follows from the equivariant
scissor relation.
By stratification we may assume
$f:X \to Y$ is a $A$-equivariant geometric bijection
of varieties.
Then by the proof of \cite[Lem 2.8]{Br2}
there is an open subset $U \subset Y$ such that
$f^{-1}(U) \to U$ is an isomorphism.
Since $f$ is $A$-equivariant
we may assume $U$ is $A$-invariant.
Replacing $X,Y$ by the complement of $U, f^{-1}(U)$
respectively and repeating the argument,
the process has to terminate
at which point we obtain
$[X,a_X] = [Y,a_Y]$ in $K_0^A(\mathrm{Var})$.
\end{proof}

We identify the groups
$K^A_0(\mathrm{Var})$ and $K^A_0(\mathrm{Sch})$
via the isomorphism \eqref{23313}.

\subsection{Power structures}
\label{Subsection_Power_Structures}
Recall from \cite{GLM} that a power structure over a commutative ring $R$ is a map
\[
(1+q R[[q]]) \times R \rightarrow (1+ q R[[q]]),
\]
denoted by $(f(q),r) \mapsto f(t)^r$,
satisfying the following 5 axioms
\begin{enumerate} 
 \item $f(q)^0 = 1$,
 \item $f(q)^1 = f(q)$,
 \item $f(q)^n\cdot g(q)^n = \left( f(q)\cdot g(q)\right)^{n}$,
 \item $f(q)^{n+m} = f(q)^{n}\cdot f(q)^{m}$,
 \item $f(q)^{nm} = \left(f(q)^n\right)^m$.
\end{enumerate}  

The power structure over the ordinary
Grothendieck ring $K_0(\mathrm{Var})$ was defined in \cite{GLM}
as follows.
Assume $S_0(\mathrm{Var})$ is the semi-subring
of $K_0(\mathrm{Var})$ spanned by effective
classes.  Let
\[
f(t) = 1+\sum_{k \geq 1} [M_k] q^k
\]
be a series in $S_0(\mathrm{Var})[[q]]$,
and $[R]$ be a class in $S_0(\mathrm{Var})$.
Then $f(t)^{[R]}$ is defined to be
the series $1+ \sum_{n \geq 1} [W_n] q^n$ with
\begin{equation}\label{def_power}
[W_n] =
\sum_{ \substack{ (k_1, k_2, \ldots ) \\ \sum_i ik_i =n }}
\left[
\Big( \big( \prod_{i} R^{k_i} \big) \setminus \triangle \Big)
\times \prod_{i} M_i^{k_i} /\prod_{i} S_{k_i}
\right],
\end{equation}
where $\triangle$ is the big diagonal in $\prod_i R^{k_i}$,
and $S_{k_i}$ acts by permuting
the corresponding $k_i$ factors
in $( \prod_i R^{k_i} )\setminus \triangle$
and $M_i^{k_i}$ simultaneously, compare \cite{GLM}.
This defines a power structure over $S_0(\mathrm{Var})$
which extends uniquely to a power structure over $K_0(\mathrm{Var})$.

We define a power structure
on the $A$-equivariant Grothendieck ring
$K^A_0(\mathrm{Var})$
by exactly the same procedure.
It only remains, given $A$ actions on $M_k$ and $R$ respectively,
to endow the classes \eqref{def_power} with $A$-actions.
The $A$-actions on $M_k$ and $R$ induce a diagonal action on each effective class   
\[
\left[
\Big( \big( \prod_{i} R^{k_i} \big) \setminus \triangle \Big)
\times \prod_{i} M_i^{k_i} /\prod_{i} S_{k_i}
\right] \,.
\]
and we let $[W_k]$ be the associated
equivariant effective class in $K^A_0(\mathrm{Var})$.
As in \cite[Thm 2.1 and 2.2]{MS} this
defines a power structure
over the semi-ring of $A$-equivariant effective classes,
which extends uniquely to $K^A_0(\mathrm{Var})$.

\subsection{Canonical decompositions and $\epsilon$-integration maps I} \label{Integration}
Let $a_X$ be an $A$-action on a variety $X$.
Let $U \subset X$ be the closed subset
of $A$-fixed points,
and let $V = X \setminus U$ be its complement.
We call the associated scissor relation
\begin{equation}\label{decomp}
[X, a_X] = [U, a_{\mathrm{triv}}]+ [V , a_{X}|_V]
\end{equation}
the canonical decomposition of $[X, a_X]$.
Since $A$ is a simple abelian group,
the induced $A$-action on $V$ has finite stabilizers
and the quotient $V/A$ is a Deligne--Mumford stack.
We define the $\epsilon$-integration map
\[
\CI \colon K^A_0(\mathrm{Var}) \rightarrow \mathbb{Q}[\epsilon]/ \epsilon^2
\]
to be the unique group homomorphism satisfying
\[ \CI( [X,a_X] ) = e(U)+  e(V/A) \cdot \epsilon \]
for every variety $X$ with canonical decompostion \eqref{decomp}.

Since stratification along stabilizers
is compatible with the scissor relation,
the canonical decomposition (\ref{decomp})
extends uniquely to all classes in $K^A_0(\mathrm{Var})$,
and the map $\CI$ is well-defined.

\begin{lem} \label{ring_structure}
The $\epsilon$-integration map $\CI$ is a ring homomorphism.
\end{lem}

\begin{proof}
Consider effective classes $[X, a_X]$ and $[Y,a_Y]$
together with their canonical decompositions
\[
[X, a_X] = [U_1, a_{\mathrm{triv}}] + [V_1, a_X|_{V_1}]
\]
and
\[
[Y, a_Y] = [U_2, a_{\mathrm{triv}}] + [V_2, a_Y|_{V_1}].
\]
The product $[X \times Y, a_{X\times Y}]$
has the canonical decomposition
\[
[X \times Y, a_{X\times Y}]
= [U_1\times U_2, a_{\mathrm{triv}}] + [V, a_V] \,,
\]
with $[V, a_V] =
[U_1\times V_2]+ [U_2 \times V_1] +[V_1\times V_2]$,
where we have suppressed the induced $A$-actions.
We have 
\[
(U_i \times V_j)/A \simeq U_i \times (V_j/A), \quad \{i,j\}=\{1,2\} \,.
\]
Since $V_1 \times V_2$ carries an $(A \times A)$-action,
the quotient $(V_1\times V_2)/A$ carries an $A$-action.
Since this action has no fixed points, $e((V_1\times V_2)/A) =0$.
Thus
\begin{align*}
\CI([X\times Y, a_{X\times Y}])
& = e( U_1 \times U_2 ) +
\epsilon \cdot \big( e(U_1) e(V_2/A) + e(U_2) e(V_1/A) \big) \\
& = \CI([X, a_X]) \cdot \CI([Y, a_Y]) \,. \qedhere
\end{align*}
\end{proof}

For $f\in 1+q\mathbb{Q}[\epsilon][[q]]$
and $g\in \mathbb{Q}[\epsilon]$
we let $f^g = e^{g\cdot \mathrm{log}(f)}$
where the logarithm is defined by the formal expansion
$\log(1+x) = -\sum_{n\geq1} (-x)^n/n$.
The associated power structure on $\BQ[\epsilon]$
is compatible with $\CI$ as follows:

\begin{lem}\label{compatible_power}
Let $Y$ be a variety,
and let $a$ be the $A$-action on $Y \times A$
by translation in the second factor. Then
\[
\CI \left( \left( \frac{1}{1-q}\right)^{[Y\times A, a]} \right)
=  \left( \frac{1}{1-q}\right)^{e(Y)\cdot \epsilon}.
\]
\end{lem}

\begin{proof}
We expand the motivic zeta function:
\[
\left( \frac{1}{1-q}\right)^{[Y\times A, a]}
= 1+ \sum_{n\geq 1}[(Y\times A)^{(n)}, a^{(n)}]\cdot q^n\,,
\]
where $(Y\times A)^{(n)}$ is the $n$-th symmetric product
of $Y \times A$ and $a^{(n)}$ is the induced $A$-action.
Hence it suffices to show
\[
e\left((Y\times A)^{(n)}/A\right)= \frac{e(Y)}{n}.
\]

Let $\pi: (Y\times A)^{(n)} \rightarrow A$ be
the composition of the projection
$(Y\times A)^{(n)} \rightarrow A^{(n)}$
and the addition map $A^{(n)}\rightarrow A$.
By \cite[Lem 28 and 29]{MS}, we have
\[
e(\pi^{-1}(0_A)) = e(Y) \cdot n^{2g-1}
\]
where $0_A \in A$ is the zero.
The stack $(Y\times A)^{(n)}/A$
is the quotient of $\pi^{-1}(0_A)$
by the group $A[n]$ of $n$-torsion points on $A$. Hence
\[
e((Y\times A)^{(n)}/A)= \frac{e(\pi^{-1}(0_A))}{n^{2g}} = \frac{e(Y)}{n}.   \qedhere
\]
\end{proof}

\subsection{Proof of Theorem \ref{Thm_degree0_Euler_Characteristic}}
By a topological argument, see \cite[Prop 2.1]{S},
the Euler characteristic $e(\Hilb^n(Y \times A)/A)$
does not depend on the choice of the abelian variety $A$.
Hence we may assume $A$ is simple.

Let $H_n = \Hilb^n( \BC^d )_0$ be the punctual
Hilbert scheme of length $n$ in $\BC^d$,
and let $[H_n]$ be its class
in $K^A_0(\mathrm{Var})$
(with the trivial $A$-action).
Let $a$ be the $A$-action on $Y \times A$
by translation in the second factor,
and let $a^{[n]}$ be the induced action
on $\mathrm{Hilb}^n(Y\times A)$.
By the stratification of $\mathrm{Hilb}^n(Y\times A)$ (compare \cite{GLM2, Goe2})  we have
\begin{equation} \label{power}
\sum_{n = 0}^{\infty} [\mathrm{Hilb}^n(Y\times A), a^{[n]}] q^n
= \left(\sum_{m = 0}^{\infty} [H_m]q^m \right)^{[Y\times A ,a]} \,.
\end{equation}

We apply the $\epsilon$-integration map to the equation (\ref{power}). Since there exits classes $[M_i]\in K^A_0(\mathrm{Var})$ with trivial $A$-actions such that
\[
\sum_{m\geq 0} [H_m]q^m = \prod_{m\geq 1} \left( \frac{1}{1-q^m} \right)^{[M_m]},
\]
by Lemmas \ref{ring_structure} and \ref{compatible_power} the integration map of the righthand side of (\ref{power}) is compatible with the power structure. It follows
\[
1+ \sum_{n \geq 1}{e(X_n)}q^n \cdot \epsilon
= \bigg( \sum_{m \geq 0} P_{d}(m) q^m \bigg)^{e(Y)\epsilon}. 
\]
Theorem \ref{Thm_degree0_Euler_Characteristic} is deduced by comparing the coefficient of $\epsilon$. \qed

\subsection{Degree 0 DT invariants}
Let $S$ be a non-singular K3 surface,
let $E$ be an elliptic curve,
and let $A$ be an abelian threefold.
\begin{cor}
\label{DT_degree0}
For all $n > 0$,
\[
\DT^{S \times E, E\mathrm{-red}}_{n,0}
= 24 \frac{(-1)^{n-1}}{n} \sum_{\ell|n} \ell^2,
\quad \ \ 
\DT_{n,0}^{A, A\mathrm{-red}}
= \frac{(-1)^{n-1}}{n} \sum_{\ell|n} \ell^2,
\]
\end{cor}
\begin{proof}
This follows by \cite{BF}, Theorem~\ref{Thm_degree0_Euler_Characteristic}, and MacMahon's
formula for 3-dimensional partitions,
\[
\sum_{m\geq 0}P_3(m)q^m= \prod_{m \geq 1}(1-t^m)^{-m}. \qedhere
\]
\end{proof}

\section{Equivariant motivic Hall algebras}
\label{Section_Equivariant_motivic_Hall_algebras}
\subsection{Overview}
Let $A$ be a simple abelian variety of dimension $g>0$.
In this section we introduce the $A$-equivariant
motivic hall algebra of~$X$
and its integration map over the ring of dual numbers.
Applying results of Joyce
we define reduced generalized Donaldson--Thomas invariants,
and prove a structure result for
reduced DT invariants generalizing
results of Toda and Bridgeland.

\subsection{Modified Grothendieck rings}
\label{Modified_Groth_rings}
The modified $A$-equivariant Grothen\-dieck group $\widetilde{K}^A_0(\mathrm{Var})$
is the $\BQ$-vector space $K^A_0(\mathrm{Var})\otimes \mathbb{Q}$
modulo the following extra relations:

\begin{enumerate}
\item[(E)] Let $X_1$, $X_2$, and $Y$ be varieties with
$A$-actions $a_1$, $a_2$, and $a_Y$ respectively.
If all $A$-actions have finite stabilizers
and $X_i\rightarrow Y$ ($i=1,2$)
are $A$-equivariant Zariski fibrations with the same fibers, then
\[
[X_1, a_1] = [X_2, a_2] \ \in \widetilde{K}^A_0(\mathrm{Var}) \,.
\]
\end{enumerate}

\begin{lem}\label{Quotient}
Under the assumptions of relation (E), we have
\[ e(X_1/A)=e(X_2/A). \]
\end{lem}
\begin{proof}
Let $W$ be the fiber of both fibrations.
The $A$-eqivariant fibration $X_i \to Y$ induces
a map $f_i : X_i/A \to Y/A$ of Deligne-Mumford stacks
with constant fiber $W$.
Hence for $i=1,2$ we have
\[
e(X_i/A) = e(W)\cdot e(Y/A) \,. \qedhere
\]
\end{proof}

The ring structure on $K^A_0(\mathrm{Var})$
induces naturally a ring structure on 
$\widetilde{K}^A_0(\mathrm{Var})$.
By Lemma \ref{Quotient} the integration map $\CI$
descends to a well-defined ring homomorphism
\begin{equation*}
\CI:  \widetilde{K}^A_0(\mathrm{Var}) \rightarrow \BQ[\epsilon].
\end{equation*}

\subsection{Preliminaries}
We will follow Bridgeland \cite{Br2}
for the discussion of
Grothendieck groups of stacks and motivic Hall algebras.
In particular, all stacks here
are assumed to be algebraic and
locally of finite type with
affine geometric stabilizers.
Geometric bijections and Zariski fibrations of stacks
are defined in \cite[Def 3.1]{Br2} and \cite[Def 3.3]{Br2} respectively.

Let $\sigma : G \times \CX \to \CX$
be a group action on a
stack $\CX$,
and let
$x : \Spec \BC \to \CX$ be a $\BC$-valued point of $\CX$.
The inertia subgroup $\mathrm{In}(x)$ of $x$ is
defined by the fiber product
\[
\begin{tikzcd}
\mathrm{In}(x) \ar{r} \ar{d} & \Spec \BC \ar{d}{x}\\
G \times \Spec \BC
\ar{r}{\sigma \circ (\mathrm{id}_G \times x)}
& \CX.
\end{tikzcd}
\]
The stabilizer group of the point $x \in \CX$ is
the fibered product
\[ \mathrm{Iso}(x) = \Spec \BC \times_{x,\CX,x} \Spec \BC\,. \]
The stabilizer group of the $G$-action at $x$
is the quotient
\[ S(x) = \mathrm{In}(x) / \mathrm{Iso}(x). \]
We refer to \cite{Rom} for a discussion of group actions on stacks.

\subsection{Equivariant Grothendieck group of stacks}
\label{Subsection_Equivariant_Grothendieck_group_of_stacks}
The following is the main definition of
Section~\ref{Section_Equivariant_motivic_Hall_algebras},
and the equivariant analog of
\cite[Defn.3.10]{Br2}.

\begin{defn} \label{defn_grothringofstacks}
Let $\mathcal{S}$ be an algebraic stack
equipped with an $A$-action $a_{\CS}$.
The relative Grothendieck group of stacks $K^A_0(\mathrm{St}/\CS)$ is defined to be the $\BQ$-vector space generated by the classes
\begin{equation*}
[\CX \xrightarrow{f} \CS, a_{\CX}]
\end{equation*}
where $\CX$ is an algebraic stack of finite type,
$a_{\CX}$ is an $A$-action on $\CX$,
and $f$ is an $A$-equivariant morphism,
modulo the following relations:

\begin{enumerate}
\item[(a)] For every pair of stacks $\CX_1$ and $\CX_2$ with $A$-actions $a_1$ and $a_2$ respectively a relation
\[
[\CX_1 \sqcup \CX_2 \xrightarrow{f_1 \sqcup f_2} \CS, a_{1}\sqcup a_2] = [\CX_1 \xrightarrow{f_1} \CS, a_1]+ [\CX_2 \xrightarrow{f_2} \CS, a_2]
\]
where $f_i$ ($i=1,2$) are $A$-equivariant.

\item[(b)]
For every commutative diagram 
\[
\begin{tikzcd}
\CX_1 \arrow{rr}{g} \arrow[swap]{dr}{f_1}& &\CX_2 \arrow{dl}{f_2}\\
& \CS & 
\end{tikzcd}    
\]
with all morphisms $A$-equivariant
and $g$ a geometric bijection
a relation
\[
[\CX_1 \xrightarrow{f_1} \CS, a_1] = [\CX_2 \xrightarrow{f_2} \CS, a_2].
\]

\item[(c)] Let $\CX_1, \CX_2, \CY$ be stacks equipped
with $A$-actions $a_1, a_2, a_Y$ respectively
satisfying one of the following conditions:
\begin{itemize}
\item[(i)] the $A$-actions $a_1, a_2, a_Y$
have stabilizers $A$ at every $\BC$-point.
\item[(ii)] the $A$-actions $a_1, a_2, a_Y$ have finite stabilizers
at every $\BC$-point.
\end{itemize}
Then for every pair
of $A$-equivariant Zariski fibrations
\[
h_1: \CX_1 \rightarrow \CY, \ \quad h_2: \CX_2 \rightarrow \CY
\]
with the same fibers and for every
$A$-equivariant morphism $\CY \xrightarrow{g} \CS$,
a relation
\[
\pushQED{\qed}
[\CX_1 \xrightarrow{g\circ h_1} \CS ,a_1]
= [\CX_2 \xrightarrow{g\circ h_2} \CS ,a_2].
\qedhere \popQED
\]
\end{enumerate}
\end{defn}

\vspace{7pt}
\noindent \textbf{Remark.}
In relation (c) the stabilizer group of all actions
must have the same type
(i.e. either finite or $A$)
for the integration maps to behave reasonable.
For example, we require the classes
\[ [ A \to \Spec \BC, a_{\mathrm{triv}} ], \quad
[A \to \Spec \BC, a_{A} ] \]
where $a_{\mathrm{triv}}$ is the trivial action
and $a_A$ is the action of $A$ on itself by translation,
to be different in $K^A_0(\mathrm{St}/\Spec \BC)$.

\subsection{Absolute Grothendieck group of stacks}
\label{Subsection_Absolute_Grothendieck_group_of_stacks}
We define the absolute $A$-equivariant Grothendieck
group of stacks by
\[
K^A_0(\mathrm{St}) = K^A_0( \mathrm{St} / \Spec \BC) \,.
\]
The product of stacks and the diagonal
action makes $K^A_0(\mathrm{St})$
a commutative ring.
Since relation~(E)
of Section \ref{Modified_Groth_rings}
is a special case of
relation~(c)
of Definition~\ref{defn_grothringofstacks},
the inclusion of the category of varieties
into the category of stacks
naturally yields a map
\begin{equation}
\widetilde{K}^A_0(\mathrm{Var}) \to K_0^A(\mathrm{St})
\label{inclusss} \end{equation}
For all $d \geq 1$ consider the classes
of the general linear group $\mathrm{GL}_d$ endowed with the trivial $A$-action,
\[ [\mathrm{GL}_d] \in \widetilde{K}^A_0(\mathrm{Var}). \]
By relation (c)
(compare \cite[3.3]{Br2})
the image of $[\mathrm{GL}_d]$ is
invertible in $K_0^A(\mathrm{St})$.
We then have the following structure
result for $K^A_0(\mathrm{St})$.

\begin{prop}\label{structure_iso}
The morphism \eqref{inclusss} induces an isomorphism
\begin{equation} \label{isomrrr}
\widetilde{K}^A_0(\mathrm{Var})[[\mathrm{GL}_d]^{-1},d \geq 1]
\, \xrightarrow{\ \simeq\ } \, K^A_0(\mathrm{St}).
\end{equation}
\end{prop}

For the proof we will require the following lemma.

\begin{lemma}\label{stack_structure}
Let $\CX$ be a stack with an $A$-action
such that every $\BC$-point of $\CX$ has
finite stabilizers.
Then there exist a variety $Y$ with an $A$-action
and a $G=\mathrm{GL}_d$ action
such that both actions commute,
and an $A$-equivariant geometric bijection
\[
f: Y/G \rightarrow \CX.
\]
\end{lemma}
\begin{proof}[Proof of Lemma \ref{stack_structure}]
Since the $A$-action on $\CX$
has finite stabilizers at $\BC$-valued points,
the quotient stack $\CX/A$ also has affine stabilizers.
By \cite[Prop 3.5]{Br2} applied to $\CX/A$
we obtain a geometric bijection
\[
g: Y/G \rightarrow \CX/A.
\]
with $Y$ a variety and $G = \mathrm{GL}_d$ for some $d$.
Form the Cartesian diagrams
\[
\begin{tikzcd}
\widetilde{W} \arrow{r} \arrow{d}&
W \arrow{r}{\tilde{g}} \arrow{d} & \CX \arrow{d} \\
Y\arrow{r}& Y/G \arrow{r}{g}& \CX/A.
\end{tikzcd}
\]
Since $g$ is a geometric bijection also
$\tilde{g}$ is a geometric bijection.
Since $\widetilde{W} \to Y$ is a $G$-equivariant $A$-torsor,
and $\widetilde{W} \to W$ is an $A$-equivariant $G$-torsor,
the induced actions of $A$ and $G$ on $\widetilde{W}$
commute. This also shows $W = \widetilde{W}/G$.

Since $\widetilde{W} \to Y$ is an $A$-torsor over the variety $Y$, we have $\widetilde{W}$ is an algebraic space, and we obtain the $A$-equivariant geometric bijection
\[ \widetilde{W}/G \rightarrow \CX. \]
Finally we need to replace the algebraic space $\widetilde{W}$ by a variety $V$. This can be achieved by using a similar stratification argument as in the proof of Lemma \ref{2333}. Since every algebraic space has an open subspace represented by an affine scheme, we may choose a subvariety $U \subset \widetilde{W}$ such that the total $(A\times G)$-orbit of $U$ is represented by a variety. Taking the complement and repeating, we can stratify $\widetilde{W}$ by $(A\times G)$-equivariant varieties $U_i \subset \widetilde{W}$. Hence we set $V= \sqcup U_i$ and obtain a geometric bijection
$V/G \rightarrow \CX$.
\end{proof}

\begin{proof}[Proof of Proposition~\ref{structure_iso}]
We construct an inverse $R$ to \eqref{isomrrr}.
Let $\CX$ be a stack with $A$-action $a_{\CX}$.
Consider the stratification
\[ \CX = \CU \sqcup \CV \]
such that the stabilizer of every $\BC$-point of $\CU$
(resp. of $\CV$)
is $A$ (a finite group).
By relations (b) and (a) we find
\[
[ \CX, a_{\CX}]
= [ \CU, a_{\CX}|_{\CU}]
+ [ \CV, a_{\CX}|_{\CV} ] \,.
\]
By relation (c, ii) with $\CY = \Spec \BC$
we have
\[ [ \CU, a_{\CX}|_{\CU}] = [ \CU, a_{\mathrm{triv}} ] \]
where $a_{\mathrm{triv}}$ is the trivial action.
Hence we may assume either
the $A$-action on $\CX$ is trivial,
or has finite stabilizers.
In the first case,
let $Y/\mathrm{GL}_d \to \CX$ be a geometric bijection
with $Y$ a variety \cite[Prop 3.5]{Br2}; then set
\[ R([\CX, a_{\mathrm{triv}}])
=
[Y, a_{\mathrm{triv}}]/[\mathrm{GL}_d]. \]
If the $A$-action on $\CX$ has finite stabilizers,
let $Y/\mathrm{GL}_d \to \CX$
be the $A$-equivariant
geometric bijection of Lemma~\ref{stack_structure};
then we set
\[ R([\CX, a_{\mathrm{triv}}])
=
[Y, a_{Y}]/[\mathrm{GL}_d] \,.
\]
It remains to check $R$
is well-defined and preserves the
relations (a,b,c).
This follows
along the lines of \cite[Lem.3.9]{Br2}
from Lemma~\ref{stack_structure},
and matching the relation (c) with the extra
relation (E) imposed on $\widetilde{K}^A_0(\mathrm{Var})$.
\end{proof}

\subsection{Hall algebras}
Let $A$ be a non-trivial simple abelian variety,
let $X$ be a non-singular projective Calabi--Yau threefold
and let 
\[ a_X : A \times X \to X \]
be a free action.
Let $\mathrm{Coh}(X)$ be the category of coherent sheaves on $X$,
and let $\CM$ be the moduli stack of objects in $\mathrm{Coh}(X)$.
The abelian variety $A$ acts on $\mathrm{Coh}(X)$
by translation by $a_X$, which induces
an $A$-action 
\[ a_{\CM} : A \times \CM \to \CM \,. \]
The equivariant motivic Hall algebra $(H^A(X), \ast)$
of $X$ is defined to be the relative Grothendieck group 
\[
H^A(X):= K^A_0(\mathrm{St}/\CM).
\]
with the product $\ast$ defined
by extensions of coherent sheaves as follows.
Let $\CM^{(2)}$ be the moduli stack of short exact sequences
\[
E_\bullet:  0 \to E_1 \to E_2 \to E_3 \to 0.
\]
The stack $\CM^{(2)}$ carries an $A$-action $a_{\CM^{(2)}}$
induced by $a_X$,
and
$A$-equivariant projections $p_i : \CM^{(2)} \rightarrow \CM$
defined by
$p_i(E_\bullet) = E_i$ for $i=1,2,3$.
Given $A$-equivariant morphisms
\[
[\CX \xrightarrow{g_1} \CM, a_{\CX}]
\quad \mathrm{and} \quad
[\CY \xrightarrow{g_2} \CM, a_{\CY}].
\]
consider the Cartesian diagram
\[ \begin{tikzcd}
\CZ \arrow{r}{\rho} \arrow{d} & \CM^{(2)} \arrow{d}{(p_1,p_3)} \\%
\CX\times\CY \arrow{r}{(g_1,g_2)}& \CM\times\CM \,.
\end{tikzcd}
\]
The morphism $\rho$ is $A$-equivariant
with respect to the natural diagonal $A$-action
$a_{\CZ}$ on $\CZ$.
We define the Hall algebra product $\ast$ by
\[
[\CX \xrightarrow{g_1} \CM, a_{\CX}] \ast [\CY \xrightarrow{g_2} \CM, a_{\CY}] = [\CZ \xrightarrow{{p_2}\circ\rho} \CM, a_\CZ].
\]
The unit of $(H^A(X), \ast)$ is the point
$[\mathrm{Spec(\BC)}\rightarrow \CM]$
corresponding to the trivial sheaf $0 \in \Coh(X)$
(together with the trivial $A$-action).

The Hall algebra $H^A(X)$ is naturally a $K^A_0(\mathrm{St})$-module
via
\[
[\CY, a_\CY]\cdot [\CZ \rightarrow \CM, a_{\CZ}] := [\CY \times \CZ \rightarrow \CM, a_{\CY \times \CZ} ],
\]
where the $A$-action $a_{\CY \times \CZ}$
is induced by the diagonal $A \rightarrow A \times A$
and the product action $a_{\CY} \times a_{\CZ}$.

\subsection{Regular classes and Poisson algebras}
Let $\BL \in \widetilde{K}^A_0(\mathrm{Var})$
be the class of the affine line (with the trivial $A$-action),
which we view also as an element in $K^A_0(\mathrm{St})$
via the morphism \eqref{inclusss}.
Consider the ring
\[
\Lambda = \widetilde{K}^A_0(\mathrm{Var})\big[ \BL^{-1},
(\BL^n + \ldots + 1)^{-1}, n \geq 1 \big]
\]
We define $H^A_{\mathrm{reg}}(X)$ to be the 
$\Lambda$-submodule of $H^A(X)$ generated by the classes
$[Z \rightarrow \CM, a_Z]$
where $Z$ is a variety with an $A$-action $a_Z$.
The elements in $H^A_{\mathrm{reg}}(X)$ are called regular.

\begin{prop}\label{Br2_Thm5.1}
The $\Lambda$-submodule of regular elements
$H^A_{\mathrm{reg}}(X)$
is closed under the Hall algebra product $\ast$,
\[
H^A_{\mathrm{reg}}(X) \ast H^A_{\mathrm{reg}}(X) \subset H^A_{\mathrm{reg}}(X),
\]
and hence a $\Lambda$-algebra.
Moreover, the quotient
\[
H^A_{\mathrm{sc}}(X):= H^A_{\mathrm{reg}}(X)/(\BL -1)H^A_{\mathrm{reg}}(X)
\]
is a commutative $\widetilde{K}^A_0(\mathrm{Var})$-algebra.
\end{prop}

We will prove Proposition \ref{Br2_Thm5.1} in Section \ref{PROOF}.

The algebra $H^A_{\mathrm{sc}}(X)$ is called
the equivariant semi-classical Hall algebra.
Identical to the non-equivariant case,
Proposition~\ref{Br2_Thm5.1} implies that the Poisson
bracket on $H^A_{\mathrm{reg}}(X)$ defined by
\[
\{f,g\}:= \frac{f\ast g - g\ast f}{ \BL -1}, \quad f,g \in H^A_{\mathrm{reg}}(X)
\]
induces a Poisson bracket on the equivariant
semi-classical Hall algebra $H^A_{\mathrm{sc}}(X)$.
Hence $(H^A_{\mathrm{sc}}(X), \ast, \{,\})$ is a Poisson algebra.

\subsection{Canonical decompositions and $\epsilon$-integration maps II} \label{integration2}
We define an integration map on the Poisson algebra
$(H^A_{\mathrm{sc}}(X), \ast, \{,\})$.

Let $K(X)$ be the Grothendieck group of coherent sheaves on $X$,
and let $\Gamma$ be the image of the Chern character map
\[
\Gamma = \mathrm{Im}(\mathrm{ch}: K(X) \rightarrow H^\ast (X, \mathbb{Q})).
\]
The Euler pairing $\chi(~~~~ , ~~~~)$ on $\mathrm{Coh}(X)$ descends to the Euler form 
\[
\chi: \Gamma \times \Gamma \rightarrow \Gamma.
\]
Consider the abelian group
\[
C^\epsilon(X):= \bigoplus_{v\in \Gamma} \BQ[\epsilon] \cdot c_v
\]
where $\epsilon^2=0$.
The product 
\begin{equation}\label{star_product}
c_{v_1} \ast c_{v_2} = (-1)^{\chi(v_1, v_2)}c_{v_1+v_2}
\end{equation}
and the Poisson bracket
\begin{equation}\label{star_bracket}
\{c_{v_1}, c_{v_2} \}=(-1)^{\chi(v_1,v_2)}\chi(v_1, v_2)c_{v_1+v_2}
\end{equation}
make $(C^\epsilon(X), \ast, \{,\})$ a Poisson algebra. 

The stack $\CM$ splits as a disjoint union of open and closed substacks
\[
\CM = \bigsqcup_{v\in \Gamma} \CM_{v}
\]
according to Chern characters in $\Gamma$. Hence the equivariant Hall algebra admits the $\Gamma$-graded decomposition
\[
H^A(\CM)= \bigoplus_{v\in \Gamma} H^A_v(\CM)
\]
where $H^A_v(\CM)$ is spanned by $A$-equivariant classes factoring through $\CM_v$. 

Parallel to (\ref{decomp})
for any $A$-equivariant effective regular class 
\[
[ Z \rightarrow \CM, a_Z] \in H^A_{\mathrm{reg}}(X)
\]
with $Z$ a variety, we define the canonical decomposition to be 
\begin{equation} \label{candecoco}
[ Z \rightarrow \CM, a_Z]
= [U \rightarrow \CM, a_{\mathrm{triv}}]+[V\rightarrow \CM, a_Z|_V]
\end{equation}
where $U$ is the closed subset formed by $A$-fixed points
and $V= Z\setminus U$.
Since $a_Z|_{V}$ has finite stablizers,
the quotient $V/A$ is a Deligne--Mumford stack.
We define the $\epsilon$-integration map 
\[
\CI: H^A_{\mathrm{sc}}(X) \rightarrow C^\epsilon(X).
\]
to be the unique group homomorphism such that
for every effective class
$[Z \xrightarrow{g} \CM, a_Z] \in H^A_v(\CM)$
with canonical decomposition~\eqref{candecoco}
we have
\begin{equation}\label{Integration2}
\CI([Z \xrightarrow{g} \CM, a_Z])
=
\left( \int_{U} g^\ast \nu_\CM \dd{e}
+ (-1)^{\mathrm{dim}A} \left( \int_{V/A} g^\ast \nu_\CM \dd{e}\right) \cdot \epsilon \right)\cdot c_v \,,
\end{equation}
where
$\nu_\CM$ is the Behrend function on $\CM$
and the second integral
is defined by
\begin{equation} \label{deffff}
\int_{V/A} g^\ast \nu_\CM \dd{e}
= \sum_{k \in \BZ}
k \cdot e\left( (g^{\ast} \nu_{\CM})^{-1}(k) / A \right) \,.
\end{equation}
Since the Behrend function is constant along $A$-orbits,
\eqref{deffff} is well-defined.

To show $\CI$ is well-defined
we need to check the morphism is compatible
with the relations (a-c) of Section~\ref{Subsection_Equivariant_Grothendieck_group_of_stacks}
restricted to regular classes.
Since we can stratify the stack $\CM$
by values of the Behrend function,
we only need to consider regular classes
$\alpha \in K^A_0(\mathrm{St}/\CM)$ over a sub-stack
\[
\CM_{\tau} \subset \CM
\]
where the Behrend function is constant.
Then by projecting $\alpha$ to an element in $K^A_0(\mathrm{St})$
and using Proposition~\ref{structure_iso}
and \cite[Lem.3.8]{Br2} we obtain
that $\CI$ is well-defined, compare \cite[7.2]{Br2}.

\begin{thm}\label{epsilon_integration}
$\CI: H^A_{\mathrm{sc}}(X) \rightarrow C^\epsilon(X)$
is a Poisson algebra homomorphism.
\end{thm}

\subsection{Proof of Proposition \ref{Br2_Thm5.1} and Theorem \ref{epsilon_integration}}\label{PROOF}
Both proofs rely on a stratification
developed in \cite[Prop 6.2]{Br2}
whose $A$-equivariant form is the following.

\begin{prop}\label{stratification}
Let $Y_1$ and $Y_2$ be varieties with $A$-actions. Assume we have $A$-equivariant morphisms
\[
f_1: Y_1 \to  \CM, \ \quad  f_2:Y_2 \to \CM,
\]
and let 
\[
\CE_i \in \Coh(Y_i \times X) ~~(i=1,2) 
\]
be the corresponding families of sheaves on $X$.
Then we can stratify $Y_1 \times Y_2$ by locally clased
$A$-invariant sub-varieties
$W \subset Y_1 \times Y_2$,
such that for each closed point $w\in W$ the vector spaces
\[
\mathrm{Ext}^k_X\left(\CE_2|_{w\times X}, \CE_1|_{w\times X} \right)
\]
have fixed dimensions $d_k(W)$, and if we form the Cartesian diagrams
\begin{equation}\label{Hall_product_diagram}
\begin{tikzcd}
Z_W \arrow{r} \arrow{d}& Z \arrow{r}{\rho} \arrow{d} & \CM^{(2)} \arrow{r}{p_2} \arrow{d}{(p_1,p_3)}& \CM \\
W\arrow{r}& Y_1\times Y_2 \arrow{r}{(f_1,f_2)}& \CM\times\CM
\end{tikzcd}
\end{equation}
then
there exist an $A$-equivariant Zariski $\BC^{d_1(W)}$-bundle
$Q \rightarrow W$ such that
\begin{equation}\label{Z_W}
Z_W \simeq \left[ Q /  \BC^{d_0(W)} \right],
\end{equation}
where $\BC^{d_0(W)}$ acts trivially.
\end{prop}

\begin{proof}
In \cite[Prop 6.2]{Br2},
the subsets $W$ are chosen to be affine.
But since extension groups form locally trivial bundles
along $A$-orbits,
we may instead also use the $A$-orbits $\cup_{a \in A} (W+a)$ in 
the proof of \cite[Prop 6.2]{Br2}.
\end{proof}

\begin{proof}[Proof of Proposition \ref{Br2_Thm5.1}]
The proof is parallel to that of \cite[Thm 5.1]{Br2},
but we spell it out here to present the general method.
Let 
\[
y_i:= [Y_i \xrightarrow{f_i} \CM, a_{Y_i}], \ \quad i=1,2
\]
be equivariant regular classes defined by $f_i$ as in Proposition \ref{stratification}. Consider the $A$-equivariant stratification as in Proposition \ref{stratification},
\[
Y_1 \times Y_2 = \bigsqcup_{j} W_j.
\]
By definition and the diagram (\ref{Hall_product_diagram}) we have
\[
y_1 \ast y_2 = \sum_{j} [Z_{W_j} \rightarrow \CM]
\]
where the morphisms is the first row of (\ref{Hall_product_diagram}).
Hence (\ref{Z_W}) yields
\begin{equation}\label{y_1_product_y_2}
    y_1 \ast y_2 = \sum_{j} \BL^{-d_0(W_j)} [Q_j \xrightarrow{g_i} \CM],
\end{equation}
where $g_i$ is the bundle induced by the universal extension,
and we have supressed all $A$-actions for clarity.
Since the right-hand side of (\ref{y_1_product_y_2}) is regular,
we have proved the first part of Proposition \ref{Br2_Thm5.1}.

We prove the second part.
Since the complement of the zero-section of $Q_j\rightarrow W_j$ is a Zariski $\mathbb{C}^\ast$-fibration over $\BP(Q_j)$
by relation (c)
of Definition~\ref{defn_grothringofstacks}
we have
\begin{equation} \label{relation_100}
[Q_j \xrightarrow{g_j} \CM]=[W_j \rightarrow \CM]+ [\BL-1][\BP(Q_j) \rightarrow \CM]
\end{equation}
with compatible $A$-actions. Hence by (\ref{y_1_product_y_2}) we have
\begin{equation}\label{y_1_product_y_2_sc}
y_1\ast y_2 =
\sum_{j}[W_j \rightarrow \CM]
= [Y_1\times Y_2 \xrightarrow{\phi} \CM] \ \quad \mathrm{mod}~~ (\BL-1),
\end{equation}
where $\phi$ is induced by
$Y_1 \times Y_2 \xrightarrow{(f_1,f_2)} \CM \times \CM$ and
\[
\CM \times \CM \rightarrow \CM, \  ([\CE_1], [\CE_2]) \mapsto [\CE_1 \oplus \CE_2]. 
\]
Since $(\ref{y_1_product_y_2_sc})$ is independent
of the order of multiplication, $\ast$ is commutative.
\end{proof}

\begin{proof}[Proof of Theorem \ref{epsilon_integration}]
For equivariant effective classes 
\[
y_i:= [Y_i \xrightarrow{f_i} \CM, a_{Y_i}], ~~~~i=1,2
\]
we need to check the product identity
\begin{equation}\label{Poisson_product}
\CI(y_1 \ast y_2)= \CI(y_1)\ast \CI(y_2)
\end{equation}
and the Poisson bracket identity
\begin{equation}\label{Poisson_bracket}
\CI(\{y_1, y_2\})= \{ \CI(y_1), \CI(y_2) \}.
\end{equation}
By stratification of $\CM$ we may assume 
$f_i$ maps into the substack $\CM_{n_i} \subset \CM$
of objects with a fixed Chern character $v_i$
such that the Behrend function on $\CM_{n_i}$
is constant with value $n_i$.
We may further assume that the effective classes $y_i$
are of one of the following types:\\[3pt]
\noindent Type 1.\ Every $\BC$-valued point on $Y_i$ is $A$-fixed with respect to $a_{Y_i}$. \\
\noindent Type 2.\ Every $\BC$-valued point on $Y_i$ has finite stabilizers with respect to $a_{Y_i}$.\\[3pt]
We follow the calculations
of \cite[Section 7.2]{Br2} to treat each case.

\vspace{2pt} \noindent
{\bf Case 1.} Both $y_1$ and $y_2$ are of Type 1.
Then the $A$-actions does not play a role
and the $\epsilon$-term does not appear.
The proof of \cite[Thm 5.2]{Br2} applies.

\vspace{2pt}
\noindent
\textbf{Case 2.} Assume $y_1$ is of Type~1 and $y_2$ is of Type~$2$.
Then by definition
\[
\CI(y_1)= n_1e(Y_1)\cdot c_{v_1}, \ \quad
\CI(y_2)= (-1)^{\mathrm{dim} A} n_2e(Y_2/A)\epsilon \cdot c_{v_2},
\]
where the quotient $Y_2/A$ is induced by the action $a_{Y_2}$.
By (\ref{y_1_product_y_2_sc})
and the first Behrend function identity in \cite[Thm 2.6]{T16}
we have
\begin{equation}\label{Formula_1}
\CI(y_1\ast y_2)
= (-1)^{\dim A} (-1)^{\chi(v_1,v_2)}
n_1n_2 \cdot e(Y_1\times Y_2/A)\epsilon \cdot c_{v_1+v_2}.
\end{equation}
We obtain the identity (\ref{Poisson_product}) by (\ref{star_product}) and
\[e(Y_1)e(Y_2/A)=e(Y_1\times Y_2/A).\]

The calculation of $\CI(\{y_1, y_2\})$ is similar.
Let $\widehat{Q}_j \xrightarrow{\hat{g}} W_j$
be the Zariski bundle induced by the extension
\[
\mathrm{Ext}^1(\CE_1, \CE_2),
\ \ ([\CE_1], [\CE_2])\in  \mathrm{Im}(f_1,f_2)\subset \CM.
\]
By the expression (\ref{y_1_product_y_2}), the relation (\ref{relation_100}), and Serre duality, we get
\begin{equation*}
\begin{aligned}
\{y_1, y_2\}= \sum_{j} \Big{(} (d_3(W_j)-d_0(W_j)) & \cdot [W_j \rightarrow \CM] \\
& + [\BP(Q_j)\rightarrow \CM]   - [\BP(\widehat{Q}_j)\rightarrow \CM] \Big{)}
\end{aligned}
\end{equation*}
where we have supressed the natural $A$-actions
on the right hand side.
The second Behrend function identity in \cite[Thm 2.6]{T16} yields
\begin{equation}\label{Formula_0}
\begin{aligned}
\CI(\{y_1, y_2 \})
&=(-1)^{\dim A}\Big{(} \sum_{j}(-1)^{\chi(v_1,v_2)}\chi(v_1,v_2)n_1n_2\cdot e(W_j/A)\epsilon \Big{)} c_{v_1+v_2}\\
& = (-1)^{\dim A}(-1)^{\chi(v_1,v_2)}n_1n_2\chi(v_1,v_2) \cdot e(Y_1\times Y_2/A)\epsilon \cdot c_{v_1+v_2}
\end{aligned}
\end{equation}
which coincides with the right-hand side of (\ref{Poisson_bracket})
by (\ref{star_bracket}).

\vspace{2pt}
\noindent
{\bf Case 3.} Both $y_1$ and $y_2$ are of Type 2.
Since $\epsilon^2 = 0$ we have
\[
\CI(y_1)\ast\CI(y_2) = \{\CI(y_1), \CI(y_2)\} =0.
\]
On the other hand, both equations (\ref{Formula_1}) and (\ref{Formula_0}) also hold in this case. The product $Y_1\times Y_2$ carries an $(A\times A)$-action with no fixed points, hence $e(Y_1\times Y_2/A) =0$, and we have
\[
\CI(y_1 \ast y_2) = \CI(\{ y_1, y_2 \}) =0. \qedhere
\]
\end{proof}

\subsection{Generalized DT invariants}
\label{N_invariants_def}
Let $\CL$ be a fixed polarization on $X$.
The slope function
\begin{equation}\label{mu_slope}
\mu_{\CL}(\CE) = \frac{\mathrm{ch}_3(\CE)}{c_1(\CL)\cdot\mathrm{ch}_2(\CE)}.
\end{equation}
defines a stability condition on the category
$\mathrm{Coh}_{\leq 1}(X)$
of sheaves with support of dimension $\leq 1$.
Let 
\[
v_{n,\beta} = (0,0, \beta,n) \in \Gamma \subset \bigoplus_{i=0}^{3}H^{2i}(X, \mathbb{Z}),
\]
be a non-zero numerical class
and consider the moduli stack 
\[ \CM_{n,\beta} \subset \CM \]
of $\mu_\CL$-semistable
sheaves in $\mathrm{Coh}_{\leq 1}(X)$
with Chern character $v_{n,\beta}$.
Since semi-stability is preserved by translation
the $A$-action on $\CM$ restricts
to an action $a_{\CM_{n,\beta}}$ on $\CM_{n,\beta}$.
We define 
\[
\delta^A_{n,\beta}
= [\CM_{n,\beta} \hookrightarrow \CM, a_{\CM_{n,\beta}}] \in H^A(X)
\]
and take the formal logarithm
\begin{equation} \label{epsilon_guy}
\epsilon^A_{n,\beta}=\sum_{
\substack{l\geq 1, 
\Sigma_{i=1}^ln_i=n,\Sigma_{i=1}^l\beta_i=\beta,\\
\frac{n_i}{\beta_i \cdot c_1(\CL)} = \frac{n}{\beta \cdot c_1(\CL)}}
}
\frac{(-1)^l}{l} \delta_{n_1,\beta_1} \ast \delta_{n_2,\beta_2}\ast \cdots \ast \delta_{n_l,\beta_l}.
\end{equation}

The following theorem is the equivariant analog of
Joyce's no pole theorem \cite[Thm.8.7]{J3},
see also \cite{BR} for a modern proof.
\begin{thm} \label{NoPoleTheorem}
The element $(\BL-1)\epsilon^A_{n,\beta} \in H^A(X)$ is regular, i.e,
\[
(\BL -1)\epsilon^A_{n,\beta} \in H^A_{\mathrm{reg}}(X).
\]
\end{thm}
\begin{proof}
We prove the Theorem by making Joyce's virtual projection operators
\cite{JStack} $A$-equivariant.
For this we work with the $A$-equivariant Hall algebra
which satisfies relations (a) and (b) of
Section~\ref{Subsection_Equivariant_Grothendieck_group_of_stacks},
but not (c). The key step here is that
every stack
\[ \CM_{n,\beta} \hookrightarrow \CM \]
admits a $A$-equivariant geometric bijection
\begin{equation} f: Y/G \to \CM_{n,\beta}\,, \label{35314} \end{equation}
where $Y$ is a variety with an $A$-action and a $G$-action
which commute. Since the virtual projection operators
are explicitly defined on $Y/G$ and $A$-equivariant,
the projection on virtual indecomposable objects
is well-defined on $\CM_{n,\beta}$
and yields an $A$-equivariant and virtual indecomposable object.
Its image in $H^A_{\mathrm{reg}}(X)$
is precisely~\eqref{epsilon_guy}
and hence $(\BL-1)\epsilon^A_{n,\beta}$ is regular.

To show \eqref{35314} we can stratify $\CM_{n, \beta}$
into a component $\CU$ where the action
has finite stabilizers, and a component $\CV$ where
the action has stabilizer group $A$ at every closed point.
The claim follows for the first component by Lemma~\ref{stack_structure},
and we only need to consider the second.
If $A$ has dimension $\geq 2$ then
there does not exist a $1$-dimensional sheaf
fixed by $A$ and $\CV$ is empty.
Hence we may assume $A$ is an elliptic curve.
Since the $A$ action on $X$ is free, the stack quotient
\[ S = X / A \]
is a non-singular proper algebraic space of dimension $2$
and hence a non-singular projective surface. Let $\pi : X \to S$ be the quotient map
and let $F$ be the class of a fiber of $\pi$.
Then $\CV$ is empty unless $\beta = dF$ and $n = 0$
for some $d>0$,
in which case let $\CN_d$ be the
moduli stack of $0$-dimensional sheaves
of length $d$ on $S$,
equipped with the trivial $A$-action.
Then pullback via $\pi$ induces
an $A$-equivariant geometric bijection onto $\CV$,
\[ \CN_d \to \CV \subset \CM_{0,dF} \]
The claim then follows from
Kresch's stratification result \cite[Prop 3.5]{Br2}
applied to $\CN_d$,
and equipping $Y$ with the trivial $A$-action.
\end{proof}

Let $(\BL-1)\epsilon^A_{n,\beta}$
denote also the projection of
\[(\BL-1)\epsilon^A_{n,\beta} \in H^A_{\mathrm{reg}}(X)\]
on the equivariant semi-classical Hall algebra $H^A_{\mathrm{sc}}(X)$.

\begin{prop}\label{N_inv}
There exists $N^{\mathrm{red}}_{n,\beta} \in \mathbb{Q}$ such that
\[
\CI((\BL-1)\epsilon^A_{n,\beta})
= -(N^{\mathrm{red}}_{n,\beta} \cdot \epsilon)\cdot c_{v_{n,\beta}}.
\]
\end{prop}
\begin{proof}
By the definition of $\CI$ we have
\[
\CI((\BL-1)\epsilon^A_{n,\beta})
= -( N_{n,\beta}+N^{\mathrm{red}}_{n,\beta} \cdot \epsilon)\cdot c_{v_{n,\beta}}.
\]
where $N_{n,\beta} \in \mathbb{Q}$
is the generalized DT invariant of \cite{Br1, T08, T16} and $N^{\mathrm{red}}_{n,\beta} \in \mathbb{Q}$.

If $\dim(A) > 1$ then no non-trivial sheaf of dimension $\leq 1$
is $A$-invariant. Hence
$N_{n,\beta}=0$ for all $(n, \beta) \neq 0$.

If $\dim(A) = 1$ then every $A$-invariant sheaf is supported on
the elliptic curve $A$
which implies
$N_{n,\beta} = 0$ by \cite[Lem. 2.11]{Tpar2}.\footnote{
We may also use
\cite[Prop 6.7]{T12} here.}
\end{proof}

\subsection{Wall-crossing formulas}
We define a reduced version of
the invariants $L_{n,\beta}$
defined in \cite{Br1, T08, T16}.
We follow the discussion in \cite[4.2]{T16}.

Let $D^b \Coh(X)$ be the bounded derived category
of coherent sheaves on $X$,
and consider the slope function
\[
\nu_{\CL}(\CE):=
\frac{c_1(\CE)\cdot{\CL}^2}{\mathrm{rk(\CE)}},
\ \ \CE \in \Coh(X).
\]
which defines a weak stability condition on $\Coh(X)$.
Let $\CA$ be the category
of complexes $I^\bullet \in D^b\Coh(X)$ satisfying the following conditions:
\begin{enumerate}
\item[(a)] $h^i(I^\bullet)=0$ if $i \neq 0,1$.
\item[(b)] All Harder--Narasimhan factors of $h^0(I^\bullet)$ have slopes $\leq 0$.
\item[(c)] All Harder--Narasimhan factors of $h^1(I^\bullet)$ have slopes $> 0$.
\end{enumerate}
By \cite[3.3]{T16} the category $\CA$ is the tilt of $\Coh(X)$ along
a torsion pair and hence abelian.
Let ${\CM}_{\CA}$ be the moduli stack
of objects in the category $\CA$,
and let
\[ a_{{\CM}_{\CA}} : A \times \CM_{\CA} \to \CM_{\CA} \]
be the natural translation action by $A$.
Then we can define the $A$-equivariant motivic Hall algebra
$H^A(\CA)$ for $\CM_\CA$ parallel to $\CM$.
By \cite[Thm 1.1]{T16}
the $\epsilon$-integration map
\begin{equation}
\CI: H^A_{\mathrm{sc}}(\CA) \rightarrow C^\epsilon(X)
\end{equation}
defined as in (\ref{Integration2})
is a homomorphism of Poisson algebras.
We refer to \cite[Sec.2]{T16}
for a detailed discussion about
how to replace $\CM$ by $\CM_{\CA}$. 

Recall the slope function (\ref{mu_slope})
and let 
$\CM^L_{n,\beta} \subset \CM_{\CA}$
be the moduli stack of objects
$I^\bullet \in \CA$ satisfying the following conditions:
\begin{enumerate}
\item[(a)] $\mathrm{ch}(I^\bullet) = u_{n,\beta} :=(1,0,-\beta, -n)$.
\item[(b)] $h^0(I^\bullet)$ is an ideal sheaf.
\item[(c)] $h^1(I^\bullet) \in \Coh_{\leq 1}(X)$ and
$\mu_\CL(\CE) \geq 0$ for every sub-sheaf $\CE \subset h^1(I^\bullet)$.
\item[(d)]
$\mathrm{Hom}(\CE [-1], I^\bullet) =0$
for any $\CE \in \Coh(X)$ with $\mu_\CL(\CE)\geq 0$.
\end{enumerate}
Let
\[
\CL^A_{n,\beta}
=[\CM^L_{n,\beta} \subset \CM_\CA, a_{{\CM}_{\CA}}|_{\CM^L_{n,\beta}}]
\in H^A(\CM_\CA)
\]
be the class defined by the moduli stack $\CM^L_{n,\beta}$,
with the $A$-action obtained by the restriction of $a_{\CM_\CA}$.
Then by \cite[4.2]{T16} the class
\[
(\BL-1) \CL^A_{n,\beta} \in H^E(\CM_\CA)
\]
is a regular element.
Define invariants
$L_{n,\beta}, L_{n,\beta}^{\mathrm{red}} \in \BQ$
by\[
\CI((\BL-1)\CL^A_{n,\beta}) = -(L_{n,\beta}+L^{\mathrm{red}}_{n,\beta}\cdot \epsilon)c_{u_{n,\beta}}.
\]
By construction $L_{n,\beta}\in \BZ$ coincide with the usual
L-invariant defined in \cite{T16}.

\begin{lem}\label{L_symmetry} For every $\beta$ we have
\begin{enumerate}
\item $L^{\mathrm{red}}_{n,\beta}
= L^{\mathrm{red}}_{-n,\beta}$ for all $n \in \BZ$,
\item $L^{\mathrm{red}}_{n,\beta}=0$
if $n\gg 0$.
\end{enumerate}
\end{lem}

\begin{proof}
By the same proof as in \cite[Sec.4]{T16}
since the dualizing functor is $A$-equivariant.
\end{proof}

In \eqref{eqn3} we defined $A$-reduced Donaldson--Thomas invariants
if the action of $A$ on the Hilbert scheme has finite stabilizers.
Here we extend the definition to the general case as follows.
For all $n, \beta$, let
\[
V_{n,\beta} \subset \Hilb^n(X, \beta)
\]
be the complement of the fixed locus of translation by $A$. Then
we define
\[
\DT^{X, A\text{-red}}_{n,\beta} =
\int_{V_{n, \beta} / A} \nu \dd{e}
\]
where $\nu : V_{n,\beta}/A \to \BZ$ is the Behrend function
on the quotient.

The following Theorem is the analog for reduced invariants
of the main structure result of Donaldson--Thomas theory
\cite{Br1, T08, T16}.
\begin{thm}\label{reduced_Toda_equation}
We have the following formula,
\begin{multline}
\label{Toda}
\sum_{n,\beta}\DT_{n,\beta}^{X,A\text{-}\mathrm{red}} q^n t^\beta \\
= \bigg( \sum_{{n>0},\beta} (-1)^{n-1}nN^{\mathrm{red}}_{n,\beta}q^nt^\beta \bigg) \cdot \bigg( \sum_{n,\beta}L_{n,\beta}q^nt^\beta \bigg)
 +\sum_{n,\beta}L^{\mathrm{red}}_{n,\beta}q^nt^\beta.
\end{multline}
\end{thm}

\begin{proof}
By a straightforward argument the identity \cite[Thm 4.8]{T16}
lifts to the $A$-equivariant motivic hall algebra.
Applying the reduced integration map $\CI$ by
Proposition~\ref{N_inv}
the left-hand side of (\ref{Toda}) is the $\epsilon$-coefficient of
\begin{equation}\label{wcformula}
\mathrm{exp}\left(\sum_{n>0,\beta} (-1)^{n-1}nN^{\mathrm{red}}_{n,\beta}\epsilon q^nt^\beta  \right)
\cdot
\left(\sum_{n,\beta}(L_{n,\beta}+L^{\mathrm{red}}_{n,\beta}\epsilon)q^nt^\beta \right). \qedhere
\end{equation}
\end{proof}

\section{Reduced DT invariants of $K3\times E$}
\label{Section_Reduced_DT_for_K3xE}
\subsection{Overview}
Let $X$ be the product
of a K3 surface $S$ and an elliptic curve $E$.
We let $E$ act on $X$ by
translation in the second factor.
Throughout the section
all reduced invariants shall be understood
as $E$-reduced invariants. In particular we write
\[
\DT^{\mathrm{red}}_{n,\beta}
=
\DT^{X, E\text{-red}}_{n,\beta} \,.
\]

\subsection{$L$-invariants}
\begin{prop}\label{vanish2}
We have $L^{\mathrm{red}}_{n,(0,d)} =0$ for all $n$ and $d \geq 0$, and
\begin{equation} \label{Linvariants}
\sum_{d \geq 0} \sum_{n \in \BZ} L_{n,(0,d)}q^nt^d = \prod_{m\geq 1}(1-t^m)^{-24}.
\end{equation}
\end{prop}

\begin{proof}
Equality \eqref{Linvariants} is proven in \cite[Prop 6.8]{T12},
hence it remains to show all
$L^{\mathrm{red}}_{n,(0,d)}$ vanish.
If $n < 0$ the Hilbert scheme
$\Hilb^n(X,(0,d))$ is empty.
If $n = 0$ we have
\[
\Hilb^0(X,(0,d)) \simeq \mathrm{Hilb}^d(S).
\]
which is invariant under the $E$-action, see
the proof of \cite[Prop 6.8]{T12}. Hence
$\DT^{\mathrm{red}}_{n, (0,d)} = 0$ for $n \leq 0$.
By Theorem~\ref{reduced_Toda_equation} we conclude
$L^{\mathrm{red}}_{n,(0,d)} = 0$ for all $n \leq 0$,
from which the result follows by Lemma~\ref{L_symmetry}.
\end{proof}

\subsection{Reduced N-invariants}
\label{Subsection_Generalized_DT}

The proof of the following result
is a modification of an argument by Toda, see \cite[Prop 6.7]{T12}.

\begin{prop} \label{Prop_N}
Let $n > 0$, $d \geq 0$ and $k = \mathrm{gcd}(n,d)$. Then
\[
N^{\mathrm{red}}_{n,(0,d)}
= \left(\frac{k}{n} \right)^2 N^{\mathrm{red}}_{k,0}
= 24 \cdot \frac{1}{n^2} \sum_{\ell|k}\ell^2 \,.
\]
\end{prop}

\begin{proof}
By \eqref{Toda} we have
\[
\DT^{\mathrm{red}}_{n,0}= (-1)^{n-1} n N^{\mathrm{red}}_{n,0}.
\]
Hence the case $d=0$ follows from Corollary~\ref{DT_degree0}.

Assume $d > 0$,
and let $\CM_{n,(0,d)}$ be the moduli space of
$\CL$-semistable sheaves
of Chern character $(0,0, (0,d), n)$, see Section~\ref{N_invariants_def}.

If $k=\mathrm{gcd}(n,d)=1$ then every
semistable sheaf $\CE$ is stable
and hence
\[ \CE = j_{\ast} \CE' \]
for a stable sheaf $\CE'$ supported
on $j : E_s\hookrightarrow X$ for some $s \in S$,
where $E_s = s \times E$.
By the classification \cite{Ati, BBDG} we conclude
\begin{equation}
\label{isofff} \CM_{n,(0,d)} \cong \CM_{1, (0,0)}
\end{equation}
where, under the identification of $E$
with its dual $\Pic^0(E)$,
the isomorphism is given by taking the
determinant on each fiber $E_s$,
\[ \CE = j_{\ast} \CE' \mapsto j_{\ast} \det(\CE') \,. \]
It remains to compare the translation action by $E$
on both sides of \eqref{isofff}.
Let $t_a: X \to X,~~(s,e)\mapsto (s,e+a)$
be the translation by an element $a \in E$.
We have 
\[
\mathrm{det}(t_a^\ast \CE)
= \mathrm{det}(\CE) \otimes \CO_{E_s}(-na)
= t_{n a}^{\ast} \mathrm{det}(\CE)
\]
Hence the isomorphism \eqref{isofff}
is $E$-equivariant with respect to
$n$-times the natural translation action
on the right hand side.
Taking into account the stabilizers group $E[n]$
of $n$-torsion points of $E$, we conclude
\[
N^{\mathrm{red}}_{n,(0,d)}
= \frac{1}{n^2} \cdot N^{\mathrm{red}}_{1,(0,0)}
= \frac{24}{n^2}
\]
which verifies the proposition in case $k=1$.

Assume $m=km_0$ and $d=kd_0$ with $\mathrm{gcd}(m_0,d_0)=1$.
Then according to \cite{Ati, BBDG}
there is no stable object in $\CM_{n,(0,d)}$.
Every semistable sheaf $\CE$ in $\CM_{n,(0,d)}$
has exactly $k$ Jordan--H\"{o}lder(JH) factors,
and each factor determines a $\BC$-valued point
in $\mathcal{M}_{n_0,(0,d_0)}$.
The universal family on $M_{n_0,(0,d_0)}\times_S X$
induces a derived equivalence
\begin{equation} \label{derived_equivalence}
D^b\Coh(X) \xrightarrow{\simeq}D^b\Coh(X)
\end{equation}
sending $[\CE] \in \CM_{n_0,(0,d_0)}$
to a skyscraper sheaf $\BC_p$ for some $p\in X$.
Hence comparing Jordan--H\"older factors we obtain
the isomorphism
\begin{equation}\label{iso2}
\CM_{n,(0,d)} \xrightarrow{\simeq} \CM_{k,(0,0)}.
\end{equation}
Applying the same argument as in the case $k=1$
to each JH-factor,
the isomorphism $\eqref{iso2}$ is $E$-equivariant
with respect to the $n_0$ times the natural
translation on $\CM_{k,(0,0)}$.
Hence
\[
N^{\mathrm{red}}_{n,(0,d)}= \frac{1}{n_0^2} N^\mathrm{red}_{k,(0,0)}
\]
and the claim follows from case $d=0$.
\end{proof}

\subsection{Proof of Theorem \ref{K3xE_thm}}
By Theorem~\ref{reduced_Toda_equation}
the reduced DT invariants are completely determined
by the $L$, the reduced $L$, and the reduced $N$ invariants.
Hence Theorem~\ref{K3xE_thm} follows from
Propositions~\ref{Prop_N} and~\ref{vanish2}.
\qed

\subsection{Proof of Theorem~\ref{Thm_DTPT}}
The Hall algebra identity of \cite[Lem 3.16]{T16}
lifts to the $A$-equivariant Hall algebra.
Applying the $\epsilon$-integration map
shows that the difference between
the generating series of reduced DT and PT invariants is 
\[
\left(\sum_{n>0} (-1)^{n-1}nN^{\mathrm{red}}_{n,0} q^n \right)\cdot \left( \sum_{n,\beta}L_{n,\beta}q^nt^\beta \right).
\]
By definition, the ordinary L-invariant $L_{n,\beta}$
vanishes if $\beta$ is not of the form $(0,d)$
since there is no $E$-fixed point in the moduli space $\CM^L_{n,\beta}$. 
Hence if $\gamma \in H_2(S,\BZ)$ is non-zero then
\[
\DT^{\mathrm{red}}_{n,(\gamma,d)}
=
\PT^{\mathrm{red}}_{n,(\gamma,d)} \,.
\]
Finally, by Propositions \ref{Prop_N} and \ref{vanish2} we obtain
\[
\pushQED{\qed}
\DT^{\mathrm{red}}_{n,(0,d)} = \PT^{\mathrm{red}}_{n,(0,d)}+
24\left[\mathrm{log}\left(\prod_{n\geq 1}\left({1-(-q)^n} \right)^n \right)\prod_{m\geq 1}(1-t^m)^{-24} \right]_{q^nt^d}.
\qedhere
\popQED
\]

\section{Reduced DT invariants for abelian 3-folds}
\subsection{Overview}
Let $B$ be an non-singular simple principally polarized
abelian surface,
let $E$ be an elliptic curve and let
\[ A = B \times E \,. \]
Here we compute the $A$-reduced DT invariant of $X$ in class
\[ (0,d) \in H_2(A, \BZ) \,. \]
By deformation invariance \cite{Gul}
this yields Theorem~\ref{Thm1}.

Since $A$ is not simple
the equivariant Hall algebra methods of
Section~\ref{Section_Equivariant_motivic_Hall_algebras}
can not be applied directly
and need to be modified.
In particular we need to account
for more complicated stabilizer groups.
For $A = B \times E$ this leads to an integration map
which takes values in the ring
\[ \BQ[ \epsilon_1, \epsilon_2 ] / ( \epsilon_1^2 = \epsilon_2^2 = 0 ) \,. \]

\subsection{Equivariant Hall algebra}
\label{Subsection_BxE_hallalgebra}
The following Lemma asserts that
all stabilizer groups of $A$ can be controlled.
\begin{lemma} \label{lemma_subgroup} Every subgroup $G$ of $A$
is of the form
(a) $G = A$, (b) $G = B \times K$, (c) $G = K' \times E$,
or (d) $G = K''$ for finite groups $K,K', K''$.
\end{lemma}
\begin{proof}
Every subgroup $G \subset A$ has finitely many connected components
all of which have the same dimension.
Let $G^\circ$ be the connected component of $G$ containing the zero.

If $G^{\circ}$ has dimension $0$, the group $G$ is of type (d).

If $G^{\circ}$ is of dimension $1$,
then $B$ simple implies that the projection
$G^{\circ} \to B$ is constant. Hence $G$ is of type (c).

If $G^{\circ}$ is of dimension $2$,
consider the projection $\pi : G^{\circ} \to E$.
If $\pi$ is non-constant it is surjective and the kernel
is a $1$-dimensional subgroup of $B$; a contradiction.
Hence $\pi$ is constant and $G$ is of the form (b).

Finally, if $G^{\circ}$ is of dimension $3$, the group $G$ is of type (a).
\end{proof}

We define the relative and absolute $A$-equivariant
Grothendieck group of stacks
parallel to Section~\ref{Section_Equivariant_motivic_Hall_algebras}.
The Definition~\ref{defn_grothringofstacks}
is identical except for relation (c). The
possible cases (i) and (ii) of stabilizers groups
have to replaced with the cases (a), (b), (c), (d) of
Lemma~\ref{lemma_subgroup}.
This yields an $A$-equivariant
motivic Hall algebra $H^A(X)$ resp. $H^A(\CA)$ with the usual properties
and structures.

We define the reduced integration map $\CI$.
For an effective regular class $[Y \rightarrow \CM, a_Y]$
with $Y$ a variety, let
\begin{equation}\label{equation0}
[Y\rightarrow \CM, a_Y]=[Y^A]+[U_1]+[U_2]+[V]
\end{equation}
be the canonical decomposition,
such that every $\BC$-point of
$Y^A, U_1, U_2, V$ has a stabilizers of type
(a), (c), (b), (d) respectively,
and we have omitted the natural
$A$-actions and the morphism to $\CM$ in the notation.
Parallel to (\ref{Integration2})
we define the $\epsilon$-integration map to be the unique group homomorphism
\begin{equation}\label{equation1}
\CI: H^A_{\mathrm{sc}}(X) \rightarrow \bigoplus_{v\in \Gamma} \mathbb{Q}[\epsilon_1, \epsilon_2]/(\epsilon_1^2, \epsilon_2^2) \cdot c_v
\end{equation}
such that for every regular effective class $[Y \rightarrow \CM_{v}, a_Y]$
\[
\begin{aligned}
\CI([Y\xrightarrow{g} \CM, a_Y])&= \Big{(} \int_{Y^A} g^\ast \nu_\CM \dd{e} + \Big{(} \int_{U_1/B} g^\ast \nu_\CM \dd{e}\Big{)}\epsilon_1 \\ &- \Big{(} \int_{U_2/E} g^\ast \nu_\CM \dd{e}\Big{)}\epsilon_2- \Big{(} \int_{V/A} g^\ast \nu_\CM \dd{e}\Big{)}\epsilon_1\epsilon_2 \Big{)}\cdot c_v \,.
\end{aligned}
\]
By the same argument as for Theorem \ref{epsilon_integration},
the $\epsilon$-integration map (\ref{equation1})
is a homomorphism of Poisson algebras. 

\subsection{Proof of Theorem~\ref{Thm1}}
\label{Subsection_ProofofThemAbelian}
Let $\epsilon^A_{n,\beta}$ be the class
in the $A$-equivariant motivic Hall algebra
defined in Section~\ref{N_invariants_def}.
Define generalized Donaldson--Thomas invariants
by the reduced integration map of Section~\ref{Subsection_BxE_hallalgebra},
\[
\CI\big((\BL-1)\epsilon^A_{n,\beta}\big)
= -\mathsf{N}^{\bullet}_{n, \beta} \cdot c_{v_{n,\beta}}.
\]
Parallel to Proposition~\ref{N_inv}
only $A$-reduced invariants are non-zero. We write
\[ \mathsf{N}^{\bullet}_{n,(0,d)} =
 N_{n,(0,d)}^{A\text{-red}}\epsilon_1 \epsilon_2 \,. \]
Define equivariant
$L$-invariants by
$\CI\big( (\BL-1) L^{A}_{n,\beta} )
=
-\mathsf{L}^{\bullet}_{n,\beta} c_{n,\beta}$
and let
\[
\DT^{\bullet}_{n,\beta}
=
\DT_{n, \beta}
+ \DT^{B\text{-red}}_{n, \beta} \epsilon_1
+ \DT^{E\text{-red}}_{n, \beta} \epsilon_2
+ \DT^{A\text{-red}}_{n, \beta} \epsilon_1 \epsilon_2 \,.
\]
As in \eqref{wcformula} an application of the
reduced integration map yields the wall-crossing
formula
\begin{equation} \label{wcformula22}
\sum_{n, \beta} \DT^{\bullet}_{n,\beta} q^n t^{\beta}
=
\exp\left(\sum_{n>0,\beta}
(-1)^{n-1}n \mathsf{N}^{\bullet}_{n,\beta} q^n t^\beta \right)
\cdot
\left(\sum_{n,\beta} \mathsf{L}^{\bullet}_{n,\beta} q^n t^\beta \right).
\end{equation}
We have
\[
\begin{cases}
\DT_{n,\beta} = 0 & \text{ if } (n,\beta) \neq 0 \\
\DT_{n,\beta}^{E\text{-red}} = 0 & \text{ for all } n,\beta \\
\DT_{n,\beta}^{B\text{-red}} = 0
& \text{ unless } \beta = (0,d), d > 0, n = 0 \\
\DT_{0,(0,d)}^{B\text{-red}}
= e(\Hilb^d(B)/B) & \text{ if } d>0,
\end{cases}
\]
which yields
\[ \sum_{n \in \BZ} L^{\bullet}_{n,(0,d)}q^n
=
\begin{cases}
e(\Hilb^d(B)/B) \epsilon_1  & \text{ if } d > 0 \\
1 & \text{ if } d = 0.
\end{cases}
\]
Picking out the $q^n t^{(0,d)}$ coefficient
in \eqref{wcformula22} hence yields
\begin{equation} \label{413413}
\DT_{n,(0,d)}^{A\text{-red}}
=
(-1)^{n-1} n {N}^{A\text{-red}}_{n, (0,d)} \,.
\end{equation}

\begin{prop}\label{Prop_N_A}
Let $n > 0$, $d \geq 0$ and $k = \mathrm{gcd}(n,d)$. Then
\[
N^{A\textup{-red}}_{n,(0,d)}
= \left(\frac{k}{n} \right)^2 N^{A\textup{-red}}_{k,0}
= \frac{1}{n^2} \sum_{\ell|k}\ell^2 \,.
\]
\end{prop}
\begin{proof}
Since the isomorphisms
\eqref{isofff} and \eqref{iso2}
are compatible with $B$-translations
the proof of Proposition \ref{Prop_N}
also shows
\[
N^{A\textup{-red}}_{n,(0,d)}
= \left(\frac{k}{n} \right)^2 N^{A\textup{-red}}_{k,0} \,.
\]
Hence the claim follows from
the $d=0$ case of \eqref{413413}
and Corollary~\ref{DT_degree0}.
\end{proof}

\begin{proof}[Proof of Theorem \ref{Thm1}]
By \eqref{413413} and Proposition~\ref{Prop_N_A}.
\end{proof}

\subsection{DT/PT correspondence}
\label{Subsection_Abelianthreefold_DTPT}
Finally we prove the DT/PT correspondence for abelian 3-folds.
If at least two of the $d_i$ are positive
in the curve class $\beta=(d_1,d_2,d_3)$
the $A$-translation on the Chow variety $\mathrm{Chow}(A,\beta)$
has no fixed point.
By a comparision of local contribution as in \cite{O1}
it follows
\[
\DT^{A\textup{-red}}_{n,\beta} = \PT^{A\textup{-red}}_{n,\beta} \,.
\]
The following theorem extends this statement to all classes.

\begin{thm}
When $d>0$ and $n>0$, we have
\[
\DT^{A\textup{-red}}_{n,(0,0,d)}=  \PT^{A\textup{-red}}_{n,(0,0,d)}.
\]
\end{thm}

\begin{proof}
By deformation invariance we may work
with the product $A=B\times E$.
Applying the $\epsilon$-integration map
of Section~\ref{Subsection_BxE_hallalgebra}
to the $A$-equivariant version of the Hall algebra identity of \cite[Lem 3.16]{T16} yields
\begin{equation*}\label{abelian_DTPT}
\sum_{n,d}\DT^\bullet_{n,(0,d)}q^nt^d= \exp\left(\sum_{n>0}
(-1)^{n-1}n \mathsf{N}^{\bullet}_{n,0} q^n t^d \right)
\cdot
\left(\sum_{n,d} \mathsf{PT}^{\bullet}_{n,(0,d)} q^n t^d \right).
\end{equation*}
where the invariants
$\DT^\bullet_{n,(0,d)}$ and $\mathsf{N}^{\bullet}_{n,0}$
are defined in Section~\ref{Subsection_ProofofThemAbelian}
and
\[
\PT^{\bullet}_{n,(0,d)}
=
\PT_{n, (0,d)}
+ \PT^{B\text{-red}}_{n, (0,d)} \epsilon_1
+ \PT^{E\text{-red}}_{n, (0,d)} \epsilon_2
+ \PT^{A\text{-red}}_{n, (0,d)} \epsilon_1 \epsilon_2 \,.
\]
By expansion using $\epsilon_1^2= \epsilon_2^2=0$ and
\[ \PT_{n,(0,d)}=0 \]
for any $(n,d) \neq 0$, the theorem is deduced.
\end{proof}


\begin{thebibliography}{10}
\bibitem{Ati} M.~Atiyah, 
{\em Vector bundles over an elliptic curve},
Proc. London Math. Soc. (3) 7 1957 414--452. 

\bibitem{BBDG}
L.~Bodnarchuk, I.~Burban, Y.~Drozd, and G.~Greuel. 
{\em Vector bundles and torsion free sheaves on degenerations of elliptic curves},
preprint. arXiv:0603261.


\bibitem{B} K.~Behrend,
{\em Donaldson--Thomas type invariants via microlocal geometry},
Ann. of Math. (2) {\bf 170} (2009), no. 3, 1307--1338. 

\bibitem{BF} K.~Behrend, B.~Fantechi, 
{\em Symmetric obstruction theories and Hilbert schemes of points on threefolds}, 
Algebra Number Theory 2 (2008), no. 3, 313--345.

\bibitem{BR}
K.~Behrend, R.~Pooya,
{\em The inertia operator on the motivic Hall algebra},
\href{https://arxiv.org/abs/1612.00372}{arXiv:1612.00372}.

\bibitem{Br1}
T.~Bridgeland, {\em Hall algebras and curve-counting invariants},
J. Amer. Math. Soc. {\bf 24} (2011), no. 4, 969--998. 

\bibitem{Br2}
T.~Bridgeland, {\em An introduction to motivic Hall algebras}, 
Adv. Math. 229 (2012), no. 1, 102--138. 

\bibitem{Bryan-K3xE} J.~Bryan, {\em The Donaldson--Thomas theory of $K3 \times E$ via the topological vertex}, \href{http://arxiv.org/abs/1504.02920}{arXiv:1504.02920}.

\bibitem{BrKo} J.~Bryan and M.~Kool,
{\em Donaldson--Thomas invariants of local elliptic surfaces
via the topological vertex},
\href{https://arxiv.org/abs/1608.07369}{arXiv:1608.07369}.

\bibitem{BOPY}
J.~Bryan, G.~Oberdieck, R.~Pandharipande, and Q.~Yin, \emph{ Curve counting on
  abelian surfaces and threefolds},
  arXiv:1506.00841.

\bibitem{Goe} L.~G{\"o}ttsche,
{\em The Betti numbers of the Hilbert scheme of points on a smooth projective surface},
Math. Ann. {\bf 286} (1990), no. 1-3, 193--207. 

\bibitem{Goe2}
L.~G{\"o}ttsche,
{\em On the motive of the Hilbert scheme of points on a surface},
Math. Res. Lett. {\bf 8} (2001), no. 5-6, 613--627. 

\bibitem{Gul}
M.~G.~Gulbrandsen, {\em Donaldson-{T}homas invariants for complexes on
  abelian threefolds},
\newblock Math. Z. {\bf 273} (2013), no. 1-2, 219--236.

\bibitem{GR}
M.~G.~Gulbrandsen and A.~Ricolfi,
{\em The Euler charateristic of the generalized Kummer scheme of an Abelian threefold},
Geom. Dedicata 182 (2016), 73--79.

\bibitem{GLM} S.~M.~Gusein-Zade, I.~ Luengo, A.~Melle-Hern\'andez, 
{\em A power structure over the Grothendieck ring of varieties}, 
Math. Res. Lett. 11 (2004), no. 1, 49--57. 

\bibitem{GLM2}S.~M.~Gusein-Zade, I.~ Luengo, A.~Melle-Hernández,
{\em A power structure over the Grothendieck ring of varieties and generating series of Hilbert schemes of points}, Michigan Math. J. 54 (2006), no. 2, 353--359. 

\bibitem{JStack}
D.~Joyce,
{\em Motivic invariants of Artin stacks and `stack functions'},
Q. J. Math. {\bf 58} (2007), no. 3, 345--392. 

\bibitem{J3} D.~Joyce,
{\em Configurations in abelian categories. III. Stability conditions and identities},
Adv. Math. {\bf 215} (2007), no. 1, 153--219.

\bibitem{MNOP}
D.~Maulik, N.~Nekrasov, A.~Okounkov, and R.~Pandharipande, 
{\em Gromov-Witten theory and Donaldson-Thomas theory I}, Compositio Math. 142 (2006), 1263--1285.

\bibitem{MS}
A.~Morrison and J.~Shen,
{\em Motivic classes of generalized Kummer schemes via relative power structures},
\href{http://arxiv.org/abs/1505.02989}{arXiv:1505.02989}.

\bibitem{O1} G.~Oberdieck, {\em On reduced stable pair invariants},
\href{http://arxiv.org/abs/1605.04631}{arXiv:1605.04631}.

\bibitem{K3xE}
G.~Oberdieck and R.~Pandharipande,
{\em Curve counting on $K3\times E$, the
Igusa cusp form $\chi_{10}$, and descendent integration},
in K3 surfaces and their moduli, C.~Faber, G.~Farkas, and G.~van der Geer, eds., Birkhauser Prog. in Math. 315 (2016), 245--278.

\bibitem{OP1}
A.~Okounkov and R.~Pandharipande, {\em Gromov-{W}itten theory, {H}urwitz
  theory, and completed cycles},
\newblock Ann. of Math. (2) {\bf 163} (2006), no. 2, 517--560.

\bibitem{OP3}
A.~Okounkov and R.~Pandharipande, {\em Virasoro constraints for target
curves}, Invent. Math. {\bf 163} (2006), no. 1, 47--108.

\bibitem{PT1}
R.~Pandharipande and R.~P.~Thomas,
\newblock {\em Curve counting via stable pairs in the derived category},
Invent. Math. {\bf 178} (2009), no. 2, 407--447.

\bibitem{PT3}
R.~Pandharipande and R.~P.~Thomas,
\newblock {\em 13/2 ways of counting curves},
\newblock in {\em Moduli spaces}, 282--333, London Math. Soc. Lecture Note Ser., {\bf 411}, Cambridge Univ. Press, Cambridge, 2014.

\bibitem{Pix}
A.~Pixton,
{\em The Gromov-Witten theory of an elliptic curve and quasimodular forms},
Senior Thesis, 2008.

\bibitem{PTVV}
T.~Pantev, B.~T\"oen, M.~Vaqu\'ie, and G. Vezzosi, 
{\em Shifted symplectic structures}, 
Publ. Math. IHES. 117 (2013), 271--328.

\bibitem{Rom}
M.~Romagny, {\em Group actions on stacks and applications}, Michigan Math. J. {\bf 53} (2005), no. 1, 209--236. 

\bibitem{S}
J.~Shen,
{\em The Euler characteristics of generalized Kummer schemes},
Math. Z. 281 (2015), no. 3--4, 1183--1189. 

\bibitem{T08}
Y.~Toda,
{\em Generating functions of stable pair invariants via wall-crossings in derived categories}, New developments in algebraic geometry, integrable systems and mirror symmetry (RIMS, Kyoto, 2008), 389--434,
Adv. Stud. Pure Math., {\bf 59}, Math. Soc. Japan, Tokyo, 2010. 

\bibitem{T10}
Y.~Toda,
{\em Curve counting theories via stable objects I: DT/PT correspondence}, 
J. Amer. Math. Soc. 23 (2010), 1119--1157.

\bibitem{Tpar2} Y.~Toda,
{\em Multiple cover formula of generalized DT invariants II: Jacobian localizations},
\href{http://arxiv.org/abs/1108.4993}{arXiv:1108.4993}.

\bibitem{T12} Y.~Toda,
{\em Stability conditions and curve counting invariants on Calabi--Yau 3-folds},
Kyoto J. Math. {\bf 52} (2012), no. 1, 1--50. 

\bibitem{T16} Y.~Toda,
{\em Hall algebras in the derived category and higher rank DT invariants},
\href{http://arxiv.org/abs/1601.07519}{arXiv:1601.07519}.

\end{thebibliography}
\end{document}